\documentclass[10pt]{article}
\usepackage{amssymb, amsmath, amsthm, verbatim, multirow, xcolor, graphicx, hyperref, longtable}

\def\Ric{\mathop{\rm Ric}}
\def\cRic{\mathop{\text{\rm R\i\makebox[0pt]{\raisebox{5pt}{\footnotesize$\circ$\;}}c}}}
\def\Riem{\mathop{\rm Rm}}

\newtheorem{theorem}{Theorem}[section]
\newtheorem{proposition}[theorem]{Proposition}
\newtheorem{lemma}[theorem]{Lemma}

\newtheorem{corollary}[theorem]{Corollary}

\title{Canonical metrics and ambiK\"ahler structures on 4-manifolds with $U(2)$ symmetry}
\author{Keaton Naff and Brian Weber}

\date{August 31, 2023}

\begin{document}
	
	\maketitle
	
	
	\abstract{
		For $U(2)$-invariant 4-metrics, we show that the $B^t$-flat metrics are very different from the other canonical metrics (Bach-flat, Einstein, extremal K\"ahler, etc).
		We show every $U(2)$-invariant metric is conformal to two separate K\"ahler metrics, leading to ambiK\"ahler structures.
		Using this observation we find new complete extremal K\"ahler metrics on the total spaces of $\mathcal{O}(-1)$ and $\mathcal{O}(+1)$ that are conformal to the Taub-bolt metric.
		In addition to its usual hyperK\"ahler structure, the Taub-NUT's conformal class contains two additional complete K\"ahler metrics that make up an ambi-K\"ahler pair, making five independent compatible complex structures for the Taub-NUT, each of which has a conformally K\"ahler (1,1) form.
	}
	
	\section{Introduction}
	
	Cohomogeneity-1 metrics with $U(2)$ symmetry have the form
	\begin{eqnarray}
		\begin{aligned}
			g\;=\;A(r)\,dr^2+B(r)\,(\eta^1)^2+C(r)\,\Big((\eta^2)^2+(\eta^3)^2\Big)
		\end{aligned} \label{EqnMetOrig}
	\end{eqnarray}
	where $\eta^1$, $\eta^2$, $\eta^3$ are the usual left-invariant covector fields on $\mathbb{S}^3$.
	Naively the topology is $\mathbb{R}\times\mathbb{S}^3$, but topological changes occur at locations where $B$ or $C$ reach zero and there could be a quotient on the $\mathbb{S}^3$ factor.
	We classify canonical metrics of this form particularly the $B^t$-flat metrics, and create some new explicit examples of canonical metrics using the ambiK\"ahler techniques of \cite{ACG16}.
	This project began as a way to develop supporting examples for other work, and treads such familiar ground that we expected few surprises.
	But we did find some surprises, two of which we feel worth reporting to the wider community.
	
	The first is how the $B^t$-flat metrics fit amongst the other canonical metrics.
	The space of extremal K\"ahler metrics is rather small---up to homothety the moduli space is 3-dimensional---and with one exception there are basically no other canonical metrics.
	Up to a choice of conformal factor, the Bach-flat metrics are a 2-parameter subspace of the extremal\footnote{We will use ``extremal'' to mean ``extremal K\"ahler.''} metrics.
	The Einstein and harmonic-curvature metrics \cite{DerdHarm} are identical, and up to conformal factors are exactly the Bach-flat metrics.
	Half-conformally flat metrics are conformally extremal, and up to conformal factors the metrics with $W^+=0$ (or $W^-=0$) form a 1-parameter subspace of the Bach-flat metrics.
	The K\"ahler-Einstein (KE) metrics and the Ricci-flat metrics are each a 1-parameter subclass of the Bach-flat metrics.
	Up to homothety there are exactly five Ricci-flat KE metrics: flat $\mathbb{R}^4$, the Eguchi-Hanson, the Taub-NUT, and two metrics with curvature singularities.
	The Taub-NUT is extraordinary; see Proposition \ref{EqnTwoRFEMetrics} and Section \ref{SubsecModTaubs}.
	
	The exception to this framework are the $B^t$-flat metrics (see \cite{GV16}), which are generally not conformally extremal.
	A $B^t$-flat metric satisfies the Euler-Lagrange equations of the functional
	\begin{equation}
		B^t\;=\;\int\vert{}W\vert^2\;+\;t\int{}s^2
	\end{equation}
	for given $t\in(-\infty,\infty]$, where $B^\infty=\int{}s^2$.
	A metric that extremizes $B^\infty$ is either scalar-flat or Einstein \cite{Besse} (the scalar-flat condition is not elliptic but the Einstein condition is second-order elliptic).
	The $B^0$ extremals are the Bach-flat metrics, and the Euler-Lagrange equations are $4^{th}$ order and underdetermined (although upon fixing a conformal gauge they become $4^{th}$ order elliptic).
	For $t\ne0,\infty$ the $B^t$ Euler-Lagrange equations are an overdetermined $8^{th}$ order system; after an appropriate reduction we find a 5-dimensional moduli space of $B^t$-flat metrics up to homothety.
	If the constant scalar curvature (CSC) condition is imposed, the CSC $B^t$-flat metrics constitute a 4-parameter family up to homothety.
	Intuitively, as $t$ varies in $[0,\infty]$, the $B^t$-flat metrics would seem to interpolate between the Einstein metrics at $t=\infty$ and the Bach-flat metrics at $t=0$.
	As we pointed out, up to conformal factors these are exactly the same class, so it would stand to reason that the $B^t$-flat metrics would stay within this class, perhaps up to conformal factors.
	But this is simply not the case, as we show in Theorem \ref{ThmNonTrivialBtFlats}.
	
	We characterize two other candidates for the status of ``canonical'' metrics: the CSC metrics and the ``half-harmonic'' (sometimes called ``weakly self-dual'') metrics.
	The half-harmonic metrics are those with either $\delta{}W^+=0$ or $\delta{}W^-=0$.
	Combining $\delta{}W^+=0$ with one other condition creates an elliptic system; this extra condition might be a scalar curvature condition such as $s=const$, a K\"ahler condition, or that $\delta{}W^-=0$.
	The case of $\delta{}W^-=0$ along with the K\"ahler condition was studied in \cite{ACG03}.
	
	The second surprise has to do with ambiK\"ahler pairs.
	Any metric (\ref{EqnMetOrig}) is automatically compatible with two complex structures which give opposite orientations---in short, each K\"ahler metric (\ref{EqnMetOrig}) is a partner in an ambiK\"ahler pair.
	From this observation we find four notable metrics: an ambiK\"ahler pair that is conformal to the classic Taub-NUT on $\mathbb{C}^2$ and an ambiK\"ahler pair that is conformal to the classic Taub-bolt.\footnote{Unlike the Taub-NUT, the Taub-bolt is non-K\"ahler so with the other two does not create any ambiK\"ahler triple.}
	The ambiK\"ahler pair conformal to the Taub-NUT are both complete extremal K\"ahler metrics, one of which has zero scalar curvature (ZSC) and is 2-ended, and the other of which is one-ended and strictly extremal.
	The two ambiK\"ahler metrics conformal to the Taub-bolt are both complete extremal metrics, and exist on two different underlying complex surfaces, $\mathcal{O}(-1)$ and $\mathcal{O}(+1)$.
	The metric on $\mathcal{O}(+1)$ is the only complete extremal K\"ahler metric, known to the authors, on the total space of any $\mathcal{O}(k)$ where $k>0$ (by contrast the Burns, Eguchi-Hanson, and LeBrun metrics are all K\"ahler metrics on $\mathcal{O}(k)$ with $k<0$).
	
	
	We first explain our computational framework.
	Solving $dz=\frac{2\sqrt{AB}}{C}dr$ for $z$, the metric (\ref{EqnMetOrig}) is
	\begin{eqnarray}
		\begin{aligned}
			g\;=\;C\left(\frac{1}{4F}dz^2+F(\eta^1)^2\,+\,(\eta^2)^2+(\eta^3)^2\right)
		\end{aligned} \label{EqnMetrInZ}
	\end{eqnarray}
	where we have abbreviated $F=\frac{B}{C}$, now a function of $z$.
	If $f=f(z)$ is any function and $\{e_1,e_1,e_3\}$ is the $\mathbb{S}^3$ frame dual to $\{\eta^1,\eta^2,\eta^3\}$, then
	\begin{equation}
		J_f\;=\;
		-2f\frac{\partial}{\partial{}z}\otimes\eta^1
		+\frac{1}{2f}e_1\otimes{}dz
		-e_2\otimes\eta^3
		+e_3\otimes\eta^2 \label{EqnJStructs}
	\end{equation}
	is always a complex stucture (where $f\ne0$); see Lemma \ref{LemmaIntegrabilityF}.
	Setting $f=\pm{}F$, the two complex structures $J^\pm=J_{\pm{}F}$ are also compatible with $g$, and produce opposite orientations.
	Their (1,1) forms are
	\begin{equation}
		\begin{aligned}
			&\omega^\pm
			\;=\;g(J^\pm\cdot,\,\cdot)
			\;=\;\pm\frac12Cdz\wedge\eta^1\,+\,C\eta^2\wedge\eta^3.
		\end{aligned} \label{EqnTwoKahlers}
	\end{equation}
	From $d\eta^i=-\epsilon^i{}_{jk}\eta^j\wedge\eta^k$ we have $d\omega^\pm=(\pm{}C+C_z)dz\wedge\eta^2\wedge\eta^3$, so a $U(2)$-invariant metric $g$ is always conformally K\"ahler, and K\"ahler when the conformal factor is chosen to be $C=C_0e^{\mp{z}}$, respectively.
	
	The following linear operators appear frequently:
	\begin{equation}
		\mathcal{L}^{+}
		=\left(-\frac12\frac{d}{dz}+1\right)\left(-\frac{d}{dz}+1\right),
		\quad
		\mathcal{L}^-
		=\left(\frac12\frac{d}{dz}+1\right)\left(\frac{d}{dz}+1\right)
	\end{equation}
	as does the $4^{th}$ order linear operator $\mathcal{L}^+\circ\mathcal{L}^-=\frac14\frac{\partial^4}{\partial{}z^4}-\frac54\frac{\partial^2}{\partial{z}^2}+1$.
	The third-order nonlinear operator $\mathcal{B}$ is also important:
	\begin{equation}
		\mathcal{B}(F,F)=
		\left(-\frac12F_{zz}+\frac32F_z+F-1\right)(\mathcal{L}^+(F)-1)
		+F_{z}\big(\mathcal{L}^+(F)\big){}_z. \label{EqnBOperatorDef}
	\end{equation}
	This operator appears messy, but a relationship exists between $\mathcal{B}$ and $\mathcal{L}^+\circ\mathcal{L}^-$.
	Namely, $\mathcal{B}$ is a first integral of the inhomogeneous operator $F\mapsto\mathcal{L}^+(\mathcal{L}^-(F))-1$, meaning that whenever $F$ solves $\mathcal{L}^+(\mathcal{L}^-(F))-1=0$ then $\mathcal{B}(F,F)$ is a constant.
	See equation (\ref{EqnBConstMotLLF}).
	We will often use $\{\sigma^0,\sigma^1,\sigma^2,\sigma^3\}$ where $\sigma^0=\frac{1}{\vert{}dz\vert}dz$, $\sigma^i=\frac{1}{\vert\eta^i\vert}\eta^i$ for the orthonormal frame corresponding to the orthogonal frame $\{dz,\eta^1,\eta^2,\eta^3\}$.
	
	\begin{proposition} \label{PropCurvatureIntro}
		The metric (\ref{EqnMetrInZ}) has scalar curvature
		\begin{equation}
			s\;=\;-4C^{-1}\left(\frac{\partial^2{}F}{\partial{}z^2}+\frac12F-2\right)
			-24C^{-\frac32}\frac{\partial}{\partial{}z}\left(F\,\frac{\partial{}C^{\frac12}}{\partial{}z}\right)
		\end{equation}
		and trace-free Ricci tensor
		\begin{equation}
			\begin{aligned}
				&\cRic
				\;=\;4FC^{-\frac12}\left(\frac{\partial^2}{\partial{}z^2}C^{-\frac12}-\frac14C^{-\frac12}\right)\cdot\left((\sigma^0)^2-(\sigma^1)^2\right) \\
				&\quad+2\left(
				C^{-\frac12}\frac{\partial}{\partial{}z}\left(F\frac{\partial}{\partial{}z}C^{-\frac12}\right)
				-C^{-1}\left(\frac12\frac{\partial^2F}{\partial{}z^2}-\frac34F+1\right)
				\right)
				\cdot\left((\sigma^0)^2+(\sigma^1)^2-(\sigma^2)^2-(\sigma^3)^2\right).
			\end{aligned} \label{EqnTfRicPre}
		\end{equation}
		The Weyl curvatures of $g$ are
		\begin{equation}
			\begin{aligned}
				W^\pm\;=\;-\frac{1}{C}(\mathcal{L}^\pm(F)-1)\left(\omega^\pm\otimes\omega^{\pm}-\frac23Id_{\bigwedge^\pm}\right)
			\end{aligned}
		\end{equation}
		and the divergences of the Weyl tensors are
		\begin{equation}
			\begin{aligned}
				\delta{}W^\pm
				\;=\;W^\pm\left(
				\nabla\log\left\vert{}e^{\pm\frac32z}(\mathcal{L}^\pm(F)-1)\sqrt{C}\right\vert,\;\cdot\,,\;\cdot\,,\;\cdot\;
				\right).
			\end{aligned}
		\end{equation}
		The Bach tensor is
		\begin{equation}
			\begin{aligned}
				Bach
				&\;=\;
				\frac{16}{3C^2}\cdot{}F\cdot\big(\mathcal{L}^-(\mathcal{L}^+(F))-1\big)\cdot\Big(-2(\sigma^1)^2\,+\,(\sigma^2)^2\,+\,(\sigma^3)^2
				\Big) \\
				&\quad\quad
				+\frac{8}{3C^2}\cdot\mathcal{B}(F,F)\cdot\Big(
				-(\sigma^0)^2-(\sigma^1)^2
				+(\sigma^2)^2+(\sigma^3)^2
				\Big).
			\end{aligned} \label{EqnBachPre}
		\end{equation}
		If the metric is K\"ahler with respect to $J^+$, so $C=C_0e^{-z}$, then the scalar curvature and Ricci form are
		\begin{equation}
			\begin{aligned}
				&s=-\frac{8}{C}\left(\mathcal{L}^+(F)-1\right), \quad\text{and} \\
				&\rho
				=
				-\frac{2}{C}\left(\mathcal{L}^+(F)-1\right)\omega^+
				-
				\frac{2}{C}\left[
				\left(-\frac12\frac{\partial}{\partial{}z}+1\right)
				\left(\frac{\partial}{\partial{}z}+1\right)F
				-1\right]\omega^-.
			\end{aligned}
		\end{equation}
	\end{proposition}
	We remark that the $U(2)$-ansatz linearlizes the Bach-flat equations $Bach=0$, reducing them to $\mathcal{L}^+\circ\mathcal{L}^-(F)-1=0$.
	The auxiliary equation $\mathcal{B}(F,F)=0$ is an algebraic restriction on initial conditions.
	
	When studying metrics---rather than just solutions of ODEs---it is useful to reduce by homothetic equivalence.
	In our case this reduces the dimension of the solution space by two: one dimension for translation in $z$ and one for multiplication of $g$ by a positive constant.
	
	\begin{proposition}[Extremal, Bach-flat, K\"ahler-Einstein metrics] \label{IntroPropExtrBachFlat}
		The metric (\ref{EqnMetrInZ}) is extremal with complex structure $J^+$ if and only if $C=C_0e^{-z}$ and $\mathcal{L}^+(\mathcal{L}^-(F))-1=0$, meaning
		\begin{equation}
			F(z)=1+\frac12C_1e^{-2z}+C_2e^{-z}+C_3e^{z}+\frac12C_4e^{2z}. \label{EqnFExtrPre}
		\end{equation}
		The metric (\ref{EqnMetrInZ}) is Bach-flat if and only if $F$ satisfies (\ref{EqnFExtrPre}) and $C_1C_4-C_2C_3=0$.
		The metric (\ref{EqnMetrInZ}) is K\"ahler-Einstein with complex structure $J^+$ if and only if $C=C_0e^{-z}$ and $F$ satisfies (\ref{EqnFExtrPre}) with $C_1=C_3=0$.
		The metric has $W^\pm=0$ if and only if $\mathcal{L}^\pm(F)-1=0$ meaning, respectively,
		\begin{equation}
			F(z)=1+C_3e^{z}+\frac12C_4e^{2z}, \quad\text{or}\quad F(z)=1+\frac12C_1e^{-2z}+C_2e^{-z}.
		\end{equation}	
		Consequently, up to homothety, the extremal metrics constitute a $3$-parameter family of metrics.
		Up to homothety and conformal transformation, the Bach-flat metrics constitute a $2$-parameter family of metrics which, up to conformal factors, is a subspace of the extremal metrics.
		Up to homothety and conformal transformation, the metrics with $W^+=0$ (or $W^-=0$) form a 1-parameter subspace of the Bach-flat metrics.
		The KE metrics also form a 1-parameter subspace of the Bach-flat metrics.
	\end{proposition}
	
	A metric is said to have \textit{harmonic curvature} if $\delta\Riem=0$, which is equivalent to $\delta{}W=0$ and $s=const$; see \cite{DerdHarm}, \cite{Bourg81}.
	In the $U(2)$-invariant case $\delta{}W=0$ already implies scalar curvature is constant.
	\begin{proposition}[Einstein and harmonic-curvature metrics]
		For the metric (\ref{EqnMetrInZ}) the following are equivalent: 1) $\delta{}W=0$, 2) $\delta\Riem=0$, 3) the metric is Einstein: $\cRic=0$, and 4) $F$ and $C$ satisfy
		\begin{equation}
			F=1+\frac12C_1e^{-2z}+C_2e^{-z}+C_3e^z+\frac12C_4e^{2z}\,,\;\;
			C=\frac{e^{-z}}{(C_5+C_6e^{-z})^2},\label{EqnEinstPre}
		\end{equation}
		with the two relations $C_1C_5-C_2C_6=0$ and $C_3C_5-C_4C_6=0$.
		Given (\ref{EqnEinstPre}), scalar curvature is the constant $s=-24(C_2C_5^2-2C_5C_6+C_3C_6^2)$.
		
		Further, a metric (\ref{EqnMetrInZ}) is Bach-flat if and only if it is conformal to an Einstein metric.
		The metric (\ref{EqnEinstPre}) is KE with respect to $J^+$ if and only if $C_6=0$ (so also $C_1=C_3=0$), and KE with respect to $J^-$ if and only if $C_5=0$ (so also $C_2=C_4=0$).
		
		Up to homothety, there is a 1-parameter family of Ricci-flat metrics.
		Up to homothety, there are exactly five Ricci-flat KE metrics: the flat metric, the Taub-NUT metric, the metric given by (\ref{EqnSuperTaubNUT}) below, the Eguchi-Hanson metric, and the metric given by (\ref{EqnSuperEguchiHanson}) below.
	\end{proposition}
	
	\begin{proposition}[CSC and half-harmonic metrics]
		The metric (\ref{EqnMetrInZ}) has scalar curvature $s=s_0$ if and only if $F$, $C$ satisfy the second order relation
		\begin{equation}
			0\;=\;
			s_0C^{\frac32}
			+4C^{\frac12}\left(\frac{\partial^2{}F}{\partial{}z^2}+\frac12F-2\right)
			+24\frac{\partial}{\partial{}z}\left(F\,\frac{\partial{}C^{\frac12}}{\partial{}z}\right),
		\end{equation}
		and has $\delta{}W^\pm=0$ if and only if $e^{\pm\frac32z}(\mathcal{L}^\pm(F)-1)\sqrt{C}$ is constant.
		
		Suppose the metric is K\"ahler with respect to $J^+$, meaning $C=C_0e^{-z}$.
		Then $\delta{}W^+=0$ if and only if $F=1+C_2e^{-z}+C_3e^{z}+\frac12C_4e^{2z}$, in which case scalar curvature is the constant $s=-24C_2/C_0$.
		Likewise $\delta{}W^-=0$ if and only if $F=1+\frac12C_1e^{-2z}+C_2e^{-z}+\frac12C_4e^{2z}$, in which case the metric is extremal and $s=-\frac{24}{C_0}(C_1e^{-z}+C_2)$.
	\end{proposition}
	
	See Section \ref{SubsecCanonicalTable} for a table of the $U(2)$-invariant canonical metrics.
	
	\begin{theorem}
		In the $U(2)$-invariant case, the space of solutions to the $B^t$-flat equations is 7-dimensional.
		Up to homothety these constitute a 5-parameter family of metrics and the CSC $B^t$-flat metrics constitute a 4-parameter family of metrics.
		When $t\ne0,\infty$, there exist CSC $B^t$-flat metrics that are not conformal to any extremal metric.
	\end{theorem}
	
	The $8^{th}$ order system for the $B^t$-flat metrics is complicated, but appears explicitly in Lemma \ref{LemmaUnReduced} below.
	In Section \ref{SubsecModTaubs} we discuss the ambiK\"ahler transform and create extremal metrics on $\mathbb{C}^2$, $\mathbb{C}^2\setminus\{(0,0)\}$ and $O(\pm1)$ conformal to the classic Taub-NUT and Taub-bolt metrics.

	\section{Properties of the Ansatz} \label{SecAnsatzProperties}
	
	The metric (\ref{EqnMetrInZ}), complex structures $J^{\pm}$, and $(1,1)$ forms $\omega^\pm$ are
	\begin{equation}
		\begin{aligned}
			g
			&\;=\;C\Big(\frac{1}{4F}dz^2+F(\eta^1)^2+(\eta^2)^2+(\eta^3)^2\Big) \\
			J^\pm{}
			&\;=\;\mp2F\frac{\partial}{\partial{}z}\otimes\eta^1
			\pm\frac{1}{2F}e_1\otimes{}dz
			-e_2\otimes\eta^3
			+e_3\otimes\eta^2 \\
			\omega^\pm{}
			&\;=\;g(J^\pm\cdot,\,\cdot)
			\;=\;\pm\frac12Cdz\wedge\eta^1\,+\,C\eta^2\wedge\eta^3.
		\end{aligned} \label{EqnMetsCxsKahls}
	\end{equation}
	We make three computations in this section.
	In Section \ref{SubSecCxStructs} we show the left-invariant complex structures $J_f$ are always integrable, and establish the K\"ahler condition for some \textit{right}-invariant complex structures as well.
	In Section \ref{SubsecCurvQuants} we compute the curvature tensors up through the Bach tensor.
	In Section \ref{SubsecTopology} we examine the topology and asymptotics which the $U(2)$ ansatz may produce, determining when manifold ends might be ALE, ALF, cusp-like, Einstein-like, or have curvature singularities, and we characterize the nut-like and bolt-like topology changes.

	\subsection{The complex structures} \label{SubSecCxStructs}
	
	Here we check the integrability of the left-invariant almost complex structures $J_f$, then study certain metric-compatible \textit{right}-invariant structures.
	\begin{lemma} \label{LemmaIntegrabilityF}
		Given any $f=f(z)\ne0$, the complex structure $J_f$ is integrable.
	\end{lemma}
	\begin{proof}
		The splitting $\bigwedge{}^1_{\mathbb{C}}=\bigwedge{}^{1,0}\oplus\bigwedge{}^{0,1}$ into $\pm\sqrt{-1}$ eigenspaces of $J_f$ gives
		\begin{equation}
			\bigwedge{}^{0,1}
			\;=\;\textit{span}{}_{\mathbb{C}}\left\{
			\frac{1}{2f}dz-\sqrt{-1}\eta^1,\;\;
			\eta^2-\sqrt{-1}\eta^3
			\right\}.
		\end{equation}
		On bases we compute
		\begin{equation}
			\begin{aligned}
				&d\left(\frac{1}{2f}dz-\sqrt{-1}\eta^1\right)
				=-2\sqrt{-1}\eta^2\wedge\eta^3
				=2\eta^2\wedge\left(\eta^2-\sqrt{-1}\eta^3\right), \\
				&d\left(\eta^2-\sqrt{-1}\eta^3\right)
				=2\eta^1\wedge\eta^3+2\sqrt{-1}\eta^1\wedge\eta^2
				=2\sqrt{-1}\eta^1\wedge\left(\eta^2-\sqrt{-1}\eta^3\right).
			\end{aligned}
		\end{equation}
		Therefore $d\bigwedge{}^{0,1}\subset\bigwedge{}^1\wedge\bigwedge{}^{0,1}=\bigwedge^{1,1}\oplus\bigwedge^{0,2}$ and we conclude that $J_f$ is integrable.
	\end{proof}
	
	\begin{lemma} \label{LemmaClosedJ}
		The complex structures $J^\pm$ are metric compatible.
		Their $(1,1)$ forms $\omega^\pm=g(J^\pm\cdot,\cdot)$ are closed if and only if $C=C_0e^{\mp{}z}$, respectively.
	\end{lemma}
	\begin{proof}
		Checking compatibility with the metric is an elementary computation (which we omit).
		From (\ref{EqnTwoKahlers}), $d\omega^\pm=0$ if and only if $C=C_0e^{\mp{}z}.$
	\end{proof}
	
	To create right-invariant complex structures and relate them to the metric (which is left-invariant) we require background coordinates.
	From the usual ``Euler coordinates'' on $\mathbb{S}^3$ comes polar coordinates on $\mathbb{R}^4\approx\mathbb{C}^2$ given by
	\begin{equation}
		(r,\psi,\theta,\varphi)\;\longmapsto\;
		\left(r\cos(\theta/2)e^{-\frac{i}{2}(\psi+\varphi)},\;
		r\sin(\theta/2)e^{-\frac{i}{2}(\psi-\varphi)}\right).
		\label{EqnCoordDefs}
	\end{equation}
	The coordinates $(\psi,\theta,\varphi)$, known historically as \textit{procession}, \textit{nutation}, and \textit{rotation}, have ranges $|\psi\pm\varphi|<2\pi$ and $\theta\in[0,\pi]$.
	The transitions between the coordinate coframe and left-invariant coframing $dz,\eta^1,\eta^2,\eta^3$ are
	\begin{equation}
		\small
		\begin{array}{ll}
			\eta^0=dz\;=\;\frac{\sqrt{F}}{2\sqrt{C}}dr & e_0=\frac{\partial}{dz}\;=\;\frac{\sqrt{F}}{2\sqrt{C}}\frac{\partial}{\partial{}r} \\
			\eta^1=\frac12(d\psi+\cos\theta\,d\varphi) & e_1=2\frac{\partial}{\partial\psi} \\
			\eta^2=\frac12(\sin\psi\,d\theta-\cos\psi\sin\theta\,d\varphi) & e_2=2\left(\cos\psi\cot\theta\frac{\partial}{\partial\psi}+\sin\psi\frac{\partial}{\partial\theta}-\cos\psi\csc\theta\frac{\partial}{\partial\varphi}\right)\\
			\eta^3=\frac12\left(\cos\psi\,d\theta+\sin\psi\sin\theta\,d\varphi\right) & e_3=2\left(
			\sin\psi\cot\theta\frac{\partial}{\partial\psi}+\cos\psi\frac{\partial}{\partial\theta}+\sin\psi\csc\theta\frac{\partial}{\partial\varphi}\right).
		\end{array}
		\label{EqnFrTransLeft}
	\end{equation}
	To create the right-invariant frames we apply quaterionic conjugation $T(z,w)=(\bar{z},-w)$ to $\mathbb{C}^2$, which changes the parameterization to
	\begin{equation}
		(r,\psi,\theta,\varphi)\;\longmapsto\;
		\left(r\cos(\theta/2)e^{\frac{i}{2}(\varphi+\psi)},\;
		-r\sin(\theta/2)e^{\frac{i}{2}(\varphi-\psi)}\right).
		\label{EqnRCoordDefs}
	\end{equation}
	In coordinates $T(r,\psi,\theta,\varphi)=(r,-\varphi,-\theta,-\psi)$.
	The left-invariant forms $\eta^i$ pull back to right-invariant forms $\bar\eta^i=T^*(\eta^i)$, and their dual vector fields are $T_*(\bar{e}_i)=e_i$.
	Under this pullback,
	\begin{equation}
		\small
		\begin{array}{ll}
			\bar\eta^0=dz\;=\;\frac{\sqrt{F}}{2\sqrt{C}}dr & \bar{e}_0=\frac{\partial}{dz}\;=\;\frac{\sqrt{F}}{2\sqrt{C}}\frac{\partial}{\partial{}r} \\
			\bar\eta^1=-\frac12(d\varphi+\cos\theta\,d\psi) & \bar{e}_1=-2\frac{\partial}{\partial\varphi} \\
			\bar\eta^2=\frac12(\sin\varphi\,d\theta-\cos\varphi\sin\theta\,d\psi) & \bar{e}_2=2\left(\cos\psi\csc\theta\frac{\partial}{\partial\varphi}+\sin\varphi\frac{\partial}{\partial\theta}-\cos\varphi\csc\theta\frac{\partial}{\partial\psi}\right)\\
			\bar\eta^3=-\frac{1}{2}\left(\cos\varphi\,d\theta+\sin\varphi\sin\theta\,d\psi\right) & \bar{e}_3=-2\left(
			\sin\varphi\cot\theta\frac{\partial}{\partial\varphi}+\cos\varphi\frac{\partial}{\partial\theta}+\sin\varphi\csc\theta\frac{\partial}{\partial\psi}\right).
		\end{array}
		\label{EqnFrTransRight}
	\end{equation}
	In the $\{\eta^i\}$, $\{\bar\eta^i\}$ bases, the map $T^*:\bigwedge^1\rightarrow\bigwedge^1$ giving $\bar\eta^i=T^*(\eta^i)$ is the matrix
	\begin{equation}
		\small
		T^*=\left(\begin{array}{cccc}
			1 & 0 & 0 & 0 \\
			0 & -\cos\theta & \cos\psi\sin\theta & -\sin\psi\sin\theta \\
			0 & \hspace{-0.05in}-\sin\theta\cos\varphi & -\cos\psi\cos\theta\cos\varphi +\sin\psi\sin\varphi & \sin\psi\cos\theta\cos\varphi+\cos\psi\sin\varphi\hspace{-0.05in} \\
			0 & \hspace{-0.05in}-\sin\theta\sin\varphi & -\cos\psi\cos\theta\sin\varphi-\sin\psi\cos\varphi & \sin\psi\cos\theta\sin\varphi-\cos\psi\cos\varphi\hspace{-0.05in}
		\end{array}\right). \label{EqnRightChoB}
	\end{equation}
	One may check directly that $T^*\in{}SO(4)$.
	Let $\sigma^i$ be the unit length forms
	\begin{equation}
		\sigma^0=\sqrt{\frac{C}{4F}}dz, \quad
		\sigma^1=\sqrt{CF}\eta^1, \quad
		\sigma^2=\sqrt{C}\eta^2, \quad
		\sigma^3=\sqrt{C}\eta^3 \label{EqnUnitForms}
	\end{equation}
	and let $\{f_0,f_1,f_2,f_3\}$ be the corresponding left-invariant frame, so $f_i=\frac{1}{\vert{}e_i\vert}e_i$.
	The left-invariant structures $J^\pm$, from above, can be expressed
	\begin{equation}
		\begin{aligned}
			&J^\pm\;=\;\mp{}f_0\otimes{}\sigma^1\pm{}f_1\otimes{}\sigma^0-f_2\otimes{}\sigma^3+f_3\otimes{}\sigma^2.
		\end{aligned}
	\end{equation}
	Under $T$ these are conjugate to the \textit{right}-invariant complex structures we call $I^-=T_*\circ{}J^+\circ{}T_*$ and $I^+=T_*\circ{}J^-\circ{}T_*$.
	Because $I^\mp$ are isomorphic to $J^\pm$ under a diffeomorphism on $M^4$ (the antipodal map on the $\mathbb{S}^3$ factor), $I^+$ and $I^-$ are integrable.
	We summarize this in the following lemma.
	\begin{lemma} \label{LemmaRighInvStructs}
		The structures $I^\pm$ are integrable, right-invariant, and $g$-compatible.
		The structures $J^+,I^+$ produce a common orientation, with corresponding $(1,1)$-forms $\omega^+,\omega^+_I\in\bigwedge^+$.
		Similarly $J^-,I^-$ produce a common orientation, and $\omega^-,\omega^-_I\in\bigwedge^-$.
		\qed
	\end{lemma}
	
	The complex structures $J^+$, $J^-$ produce a very flexible array of possible K\"ahler metrics, as $F$ may be chosen freely and only $C$ is constrained.
	By contrast, the K\"ahler condition on the $\omega^\pm_I$ is far more restrictive.
	This is because the left-action of $SU(2)$ fixes $g$ but permutes $I^\pm$ among an $\mathbb{S}^2$ worth of complex structures, which in turn means that $d\omega_I^\pm=0$ forces $\omega_I^\pm$ to be not just K\"ahler but a K\"ahler representative in a hyperK\"ahler structure.
	In particular the K\"ahler condition forces $\Ric=0$.
	\begin{proposition} \label{PropIntegrability}
		Letting $\omega_I^-=g(I^-\cdot,\cdot)$, then $d\omega_I^-=0$ if and only if
		\begin{equation}
			F\;=\;\left(1+C_1e^{z}\right)^2
			\quad\text{and}\quad
			C\;=\;
			\frac{C_0e^{z}}{\left(1+C_1e^{z}\right)^2}. \label{EqnHyperKCoefs}
		\end{equation}
		In this case the metric $g$ is Ricci-flat.
		Replacing $z$ by $-z$ (\ref{EqnHyperKCoefs}), the same is true for $I^+$.
	\end{proposition}
	\begin{proof}
		We may compute $d\omega_I^-$ explicitly using the matrices for $T^*$ in (\ref{EqnRightChoB}) and its inverse-transpose $T_*$.
		The computation is tedious but completely elementary, and works out to be
		\begin{equation}
			\begin{aligned}
				*d\omega_1
				&\;=\;
				\frac{2}{\sqrt{C}}\Bigg(
				\cos\theta\left((-2+F^{\frac12})+F^{\frac12}\frac{\partial}{\partial{}z}\log\,C\right)
				\eta^1 \\
				&\quad\quad\quad
				-F^{-\frac12}\sin\theta\cos\psi		\left(2F^{\frac12}-2F\frac{\partial}{\partial{}z}\log\,C-\frac{\partial}{\partial{}z}F\right)\eta^2 \\
				&\quad\quad\quad
				-F^{-\frac12}\sin\theta\sin\psi
				\left(2F^{\frac12}-2F\frac{\partial}{\partial{}z}\log\,C-\frac{\partial}{\partial{}z}F\right)\eta^3\Bigg).
			\end{aligned}
		\end{equation}
		Setting this to zero gives the partially decoupled system
		\begin{equation}
			\begin{aligned}
				&\frac{\partial}{\partial{}z}F^{\frac12}\;=\;\left(-1+F^{\frac12}\right),
				\quad
				\frac{\partial}{\partial{}z}\log\,C\;=\;\left(-1+2F^{-\frac12}\right)
			\end{aligned}
		\end{equation}
		which has general solution $F=\left(1+C_1e^{z}\right)^2$, $C=\frac{C_0e^{z}}{\left(1+C_1e^{z}\right)^2}$.
		Ricci-flatness follows from the general fact (see \cite{Besse}) that any hyperK\"ahler metric is Ricci flat, or, more explicitly, from Proposition \ref{PropEinstCondition} below.
	\end{proof}
	
	Proposition \ref{PropIntegrability} gives a two parameter family of solutions.
	Therefore up to homothety we have, not a 2-parameter family, but exactly two metrics.
	\begin{proposition} \label{EqnTwoRFEMetrics}
		Up to homothety, there are exactly two metrics $g$ of the form (\ref{EqnMetrInZ}) for which $I^-$ is a K\"ahler structure.
		The first is
		\begin{equation}
			F\;=\;\left(1-e^z\right)^2 \quad and \quad
			C\;=\;\frac{e^z}{(1-e^z)^2}
		\end{equation}
		which on $z\in(0,\infty]$ is the classic Taub-NUT metric.
		It has an ALF end at $z=0$ and a nut at $z=+\infty$.
		The second is
		\begin{equation}
			F\;=\;\left(1+e^z\right)^2 \quad and \quad
			C\;=\;\frac{e^z}{(1+e^z)^2}. \label{EqnSuperTaubNUT}
		\end{equation}
		This has a nut at $z=-\infty$ and curvature singularity at $z=+\infty$.
	\end{proposition}
	For information about the Taub-NUT metric see Section \ref{SubsecModTaubs}.
	For an analysis of the nut-like topological change see \S\ref{SubSubsecBoltsNuts} and for ALF ends see \S\ref{SubSubSecALEALFAndCusp}.
	To verify the claim that (\ref{EqnSuperTaubNUT}) has a curvature singularity as $z\rightarrow+\infty$, one may use (\ref{EqnWeylNorm}) below to find $\vert{}W^+\vert^2=384(-1+e^z)^6$.

	\subsection{Curvature quantities} \label{SubsecCurvQuants}
	
	A useful computational tool comes from placing the metric (\ref{EqnMetsCxsKahls}) into LeBrun ansatz form \cite{LeB91}.
	Referring to the polar coordinates of (\ref{EqnCoordDefs}), from $(r,\varphi,\theta,\psi)$ we change to $(Z,\tau,x,y)$ where $x=\log\tan\frac{\theta}{2}$, $y=\varphi$, $\tau=\psi$, and $Z$ solves $dZ=\frac14Cdz$.
	Then $(\eta^2)^2+(\eta^3)^2=\frac14(d\theta^2+\sin^2\theta\,d\varphi^2)=\frac{1}{4\cosh^2x}(dx^2+dy^2)$ and
	\begin{equation}
		g\;=\;
		\frac{C}{4\cosh^2x}(dx^2+dy^2)\;+\;\frac{FC}{4}\left(d\tau-\tanh(x)dy\right)^2+\frac{4}{FC}dZ^2. \label{EqnLeBrunForm}
	\end{equation}
	Written this way, the metric (\ref{EqnLeBrunForm}) is precisely in the form of Proposition 1 of \cite{LeB91}---the LeBrun ansatz---where $w=\frac{4}{FC}$ and $e^u=\frac{FC^2}{16\cosh^2x}$.
	The complex structures in these coordinates are
	\begin{equation}
		J^\pm(dZ)=\mp2FC\eta^1, \quad
		J^\pm(dx)=-dy \label{EqnJStructsLeBrun}
	\end{equation}
	and we record the useful fact that $\eta^2\wedge\eta^3=\frac{1}{4\cosh^2(x)}dx\wedge{}dy$.
	
	\begin{proposition}[Ricci Curvature in the K\"ahler case] \label{PropRicForm}
		If $g$ is K\"ahler with respect to $J^+$, its Ricci form $\rho=Ric(J\cdot,\cdot)$ and scalar curvature are
		\begin{eqnarray}
			\rho
			&=&
			-\frac{2}{C}\left(\mathcal{L}^+(F)-1\right)\omega^+
			-
			\frac{2}{C}\left[
			\left(-\frac12\frac{\partial}{\partial{}z}+1\right)
			\left(\frac{\partial}{\partial{}z}+1\right)F
			-1\right]\omega^-, \\ \label{EqnRhoSeparation}
			s
			&=&-\frac{8}{C}\left(\mathcal{L}^+(F)-1\right). \label{EqnScalCurvComp}
		\end{eqnarray}
	\end{proposition}
	\begin{proof}
		Setting $C=C_0e^{-z}$ we follow the computation in \cite{LeB91}.
		From that paper, the Ricci form is $\rho=-i\partial\bar\partial{}u=\frac12d(Jdu)$ where in our case $u=\log(FC^2)-\log(16\text{cosh}^2(x))$, as we found in (\ref{EqnLeBrunForm}).
		Using coordinates $(z,\tau,x,y)$ (specifically using $z$, not $Z$ from (\ref{EqnLeBrunForm})), we have $J(dz)=-2F\eta^1$ and $J(dx)=-dy$ from (\ref{EqnJStructs}) and (\ref{EqnJStructsLeBrun}).
		Using also $dx\wedge{}dy=4\cosh^2(x)\eta^2\wedge\eta^3$ and $d\eta^1=-2\eta^2\wedge\eta^3$,
		\begin{equation}
			\begin{aligned}
				&u\;=\;\log{}F-2z+2\log{}C_0-2\log(4\cosh\,x) \\
				&du\;=\;(F_zF^{-1}-2)dz\,-\,2\tanh(x)dx \\
				&Jdu\;=\;(-2F_z+4F)\eta^1\,+\,2\tanh(x)dy \\
				&dJdu=(-2F_{zz}+4F_z)dz\wedge\eta^1+(-4F_z-8F+8)\eta^2\wedge\eta^3
			\end{aligned} \label{EqnRicFormComp}
		\end{equation}
		From (\ref{EqnMetsCxsKahls}), $dz\wedge\eta^1=C^{-1}(\omega^+-\omega^-)$ and $\eta^2\wedge\eta^3=\frac12C^{-1}(\omega^++\omega^-)$.
		Therefore
		\begin{equation}
			\footnotesize
			\rho
			=
			\frac{2}{C}\left(-\frac12F_{zz}+\frac32F_z-F+1\right)\omega^+
			+\frac{2}{C}\left(\frac12F_{zz}-\frac12F_z-F+1\right)\omega^-
		\end{equation}
		as claimed.
		Scalar curvature for any K\"ahler metric is $s=2*(\omega^+\wedge\rho)$, so (\ref{EqnRhoSeparation}) along with the facts $\omega^+\wedge\omega^-=0$ and $*(\omega^+\wedge\omega^+)=2$ gives (\ref{EqnScalCurvComp}).
	\end{proof}
	
	\begin{proposition}[Ricci curvature, general case] \label{PropRicci}
		Scalar curvature is
		\begin{equation}
			s\;=\;-4C^{-1}\left(\frac{\partial^2F}{\partial{z}^2}+\frac12F-2\right)
			-24C^{-\frac32}\frac{\partial}{\partial{}z}\left(F\frac{\partial}{\partial{}z}C^{\frac12}\right). \label{EqnScalGeneral}
		\end{equation}
		Using the unit frames $\sigma^i$ of (\ref{EqnUnitForms}) the trace-free Ricci curvature is
		\begin{equation}
			\begin{aligned}
				\cRic
				&\;=\;4FC^{-\frac12}\left(\frac{\partial^2}{\partial{}z^2}C^{-\frac12}-\frac14C^{-\frac12}\right)\cdot\left((\sigma^0)^2-(\sigma^1)^2\right) \\
				&\quad+2\left(
				C^{-\frac12}\frac{\partial}{\partial{}z}\left(F\frac{\partial}{\partial{}z}C^{-\frac12}\right)
				-C^{-1}\left(\frac12\frac{\partial^2F}{\partial{}z^2}-\frac34F+1\right)
				\right)
				\cdot\left((\sigma^0)^2+(\sigma^1)^2-(\sigma^2)^2-(\sigma^3)^2\right).
			\end{aligned} \label{EqnTfRic}
		\end{equation}
	\end{proposition}
	\begin{proof}
		We use the conformal change formulas from \cite{Besse}.
		The scalar curvature (\ref{EqnScalGeneral}) follows from (\ref{EqnScalCurvComp}) along with the formula $\tilde{s}=U^{-2}(s-6U^{-1}\triangle_g{}U)$ when $\tilde{g}=U^{2}g$.
		In the K\"ahler metric where $C=e^{-z}$, the Laplacian $\triangle_g$ acting on any $U=U(z)$ is $\triangle_g{}U=4e^{2z}\frac{\partial}{\partial{}z}\left(e^{-z}F\frac{\partial{U}}{\partial{}z}\right)$.
		To obtain (\ref{EqnScalGeneral}), use $U=e^{\frac12z}C^{\frac12}$.
		
		To compute $\cRic$, again we start with the K\"ahler case; (\ref{EqnRhoSeparation}) gives
		\begin{equation}
			\cRic{}_g
			\;=\;2e^z\left(\frac12F_{zz}-\frac12F_z-F+1\right)\left(
			-(\sigma^0)^2-(\sigma^1)^2+(\sigma^2)^2+(\sigma^3)^2
			\right) \label{EqnTfRicRecorded}
		\end{equation}
		The conformal change formula for the trace-free Ricci is $\cRic_{\tilde{g}}=\cRic_g+2U(\nabla^2_gU^{-1}-\frac14(\triangle_gU^{-1})g)$.
		Then
		\begin{equation}
			\begin{aligned}
				2U\Big(\nabla^2_gU^{-1}-\frac14(\triangle_gU^{-1})g\Big)
				=
				&-4UF(e^z(U^{-1})_z)_z\left(-(\sigma^0)^2+(\sigma^1)^2\right) \\
				&-2U(e^zF\,(U^{-1})_z)_z\big(-(\sigma^0)^2-(\sigma^1)^2+(\sigma^2)^2+(\sigma^3)^2\big)
			\end{aligned} \label{EqnConfAddTfRic}
		\end{equation}
		so with $U=e^{\frac12z}C^{\frac12}$ we add (\ref{EqnConfAddTfRic}) to (\ref{EqnTfRicRecorded}) to give (\ref{EqnTfRic}).
	\end{proof}
	
	\begin{proposition} \label{PropWeylForm}
		The metric (\ref{EqnMetsCxsKahls}) has Weyl curvatures
		\begin{equation}
			\begin{aligned}
				&W^\pm
				\;=\;-C^{-1}\left(\mathcal{L}^\pm(F)-1\right)\left(\omega^\pm\otimes\omega^\pm-\frac23Id_{\bigwedge^\pm}\right).
			\end{aligned} \label{EqnWeyls}
		\end{equation}
	\end{proposition}
	\begin{proof}
		We use Derdzinski's Theorem (\cite{Derd}, section 3, proposition 2) to find $W^+$ in the K\"ahler case, then conformally change to the arbitrary case.
		By Derdzinski's Theorem $W^+=\frac{s}{12}\left(\frac32\omega\otimes\omega-Id_{\bigwedge^+}\right)$ where $\omega$ is a K\"ahler form.
		When $C=e^{-z}$, $\omega^+$ is K\"ahler and Proposition \ref{PropRicForm} gives
		\begin{equation}
			W^+=-\frac23\,e^z\,(\mathcal{L}^+(F)-1)\left(\frac32\omega^+\otimes\omega^+-Id_{\bigwedge^+}\right). \label{EqnWPlusKahler}
		\end{equation}
		Conformally changing from $C=e^{-z}$ to any $C=C(z)$ gives (\ref{EqnWeyls}).
		Computing $W^-$ is the same, after setting $C=e^{z}$ so $\omega^-$ rather than $\omega^+$ is a K\"ahler form.
	\end{proof}
	\begin{proposition} \label{PropDelWpm}
		The metric (\ref{EqnMetsCxsKahls}) has $\delta{}W^+$ and $\delta{}W^-$ given by
		\begin{equation}
			\begin{aligned}
				&\delta{}W^\pm
				\;=\;W^\pm\left(\nabla\log\left\vert{}e^{\pm\frac32z}(\mathcal{L}^\pm(F)-1)\sqrt{C}\right\vert,\;\cdot\,,\;\cdot\,,\;\cdot\;\right).
			\end{aligned}
		\end{equation}
	\end{proposition}
	\begin{proof}
		We begin again by conformally changing the metric so it is K\"ahler.
		By Lemma (\ref{PropIntegrability}) the metric $\tilde{g}=e^{-z}C^{-1}g$ is K\"ahler and the form $\widetilde{\omega}=\tilde{g}(J^+\cdot,\cdot)$ is closed.
		Then $\tilde\delta\widetilde{\omega}=-*d\widetilde{\omega}=0$ so also $\tilde\delta(\widetilde{\omega}\otimes\widetilde{\omega})=0$, and $\tilde\delta(Id_{\bigwedge^+})=0$ because $Id_{\bigwedge^+}$ is covariant-constant.
		Therefore (\ref{EqnWeyls}) gives
		\begin{equation}
			\begin{aligned}
				\tilde\delta\,\widetilde{W}^+(\cdot,\cdot,\cdot)
				&\;=\;\tilde\delta\left(-e^{z}(\mathcal{L}^+(F)-1)\left(\widetilde{\omega}\otimes\widetilde{\omega}-\frac23Id_{\bigwedge^+}\right)\right)(\cdot,\cdot,\cdot) \\
				&\;=\;-\left(\widetilde{\omega}\otimes\widetilde{\omega}-\frac23Id_{\bigwedge^+}\right)\left(\widetilde{\nabla}(e^{z}(\mathcal{L}^+(F)-1)),\;\cdot\,,\;\cdot\,,\;\cdot\,\right) \\
				&\;=\;\widetilde{W}^+\left(\widetilde{\nabla}\log\left\vert{}e^{z}(\mathcal{L}^+(F)-1)\right\vert,\;\cdot\,,\;\cdot\,,\;\cdot\,\right) \\
				&\;=\;W^+\left(\nabla\log\left\vert{}e^{z}(\mathcal{L}^+(F)-1)\right\vert,\;\cdot\,,\;\cdot\,,\;\cdot\,\right).
			\end{aligned} \label{EqnDelWpKahler}
		\end{equation}
		Derdzinski's conformal change formula, equation (19) of \cite{Derd}, is
		\begin{equation}
			\tilde\delta\,\widetilde{W}^+
			\;=\;
			\delta{}W^+-\frac12W^+\left(\nabla\log\left(e^{z}C\right),\;\cdot\,,\;\cdot\,,\;\cdot\,\right) 	\label{EqnDelWp}
		\end{equation}
		so changing the metric back with conformal factor $e^{z}C$, (\ref{EqnDelWpKahler}) and (\ref{EqnDelWp}) give
		\begin{equation}
			\delta{}W^+=W^+\left(\nabla\log\left\vert{}e^{\frac32z}(\mathcal{L}^+(F)-1)\sqrt{C}\right\vert,\;\cdot\,,\;\cdot\,,\;\cdot\,\right).
		\end{equation}
		The argument for $\delta{}W^-$ is entirely the same, after conformally changing so that $\tilde\omega^-$ not $\tilde\omega$ is closed.
	\end{proof}
	
	There are two basic ways to compute the Bach tensor.
	We could use direct computation using perhaps (24) of \cite{Derd} or in the K\"ahler case that $Bach(J\cdot,\cdot)$ is a multiple of $(s\rho+2\sqrt{-1}\partial\bar\partial{}s)_0$; see (39) of \cite{Derd}, (21) of \cite{ACG03}, or Lemma 6 of \cite{CLW}.
	Alternatively we could compute the variation of $\int|W|^2dVol_4$ directly.
	We choose to do the latter, because it helps elucidate the nature of the quadratic functional $\mathcal{B}(F,F)$: it comes from varying the diffeomorphism gauge, and its constancy along solutions of $\mathcal{L}^+\circ\mathcal{L}^-(F)-1=0$ is an expression of the second Bianchi identity.
	
	From Proposition \ref{PropWeylForm} we easily compute $\vert{}W^\pm\vert^2$ and $\vert{}W^\pm\vert^2dVol$.
	These are
	\begin{equation}
		\begin{aligned}
			&\vert{}W^\pm\vert^2=\frac{32}{3C^2}\left(\mathcal{L}^\pm(F)-1\right)^2, \quad \text{and} \\
			&\vert{}W^\pm\vert^2dVol
			=\frac{16}{3}\left(\mathcal{L}^\pm(F)-1\right)^2dz\wedge\eta^1\wedge\eta^2\wedge\eta^3.
		\end{aligned} \label{EqnWeylNorm}
	\end{equation}
	With this in hand, we can compute the Bach tensor by carrying out variations of the metric and observing the changes in $\int\vert{}W^+\vert^2dVol$.
	Only two independent variations are possible while maintaining $U(2)$ symmetry and the conformal gauge (maintaining the conformal gauge amounts to using only trace-free variations in the metric).
	Briefly, the reason this is true is that the metric as expressed in the $r$-coordinate (\ref{EqnMetOrig}) has three independent parameters $A$, $B$, and $C$ that may vary, but maintaining the conformal gauge creates a relation among these so there are really only two parameters.
	
	The first of our two variations will be to vary $F=B/C$ via
	\begin{equation}
		F\mapsto{}F+sf \label{EqnVarFirst}
	\end{equation}
	while keeping $C$ unchanged so the conformal gauge is preserved.
	Below, we'll see that stabilizing this variation is equivalent to the linear equation $\mathcal{L}^+\circ\mathcal{L}^-(F)-1=0$.
	
	For the second of our variations, we create a diffeomorphism flow that alters the coordinate $z$, while fixing both $F$ and the conformal gauge.
	Stabilizing this variation leads to the non-linear equation $\mathcal{B}(F,F)=0$.
	Such a diffeomorphism flow can be expressed
	\begin{equation}
		L_{\frac{d}{dt}}dz\;=\;\alpha{}dz, \quad\quad
		L_{\frac{d}{dt}}\frac{\partial}{\partial{}z}\;=\;-\alpha\frac{\partial}{\partial{}z} \label{EqnVarZ}
	\end{equation}
	where time $t$ is the variation parameter, $\alpha=\alpha(z)$ is a function, and $L_{\frac{d}{dt}}$ is the Lie derivative along the flow.
	For the variation in $z$ to occur in a compact region requires both that $\alpha$ have compact support and $\int_{-\infty}^\infty\alpha{}dz=0$, in which case $\alpha$ has an antiderivative with compact support.
	We define $\mathcal{A}=\mathcal{A}(z)$ to be the unique function so that
	\begin{equation}
		\text{$\mathcal{A}'(z)\;=\;\alpha(z),$ and $\mathcal{A}$ has compast support.}
	\end{equation}
	From (\ref{EqnVarZ}) we have the variational field $\frac{d}{dt}=\mathcal{A}\frac{\partial}{\partial{}z}$.
	The function $F$ changes with $z$, so to force $F$ to remain constant along the diffeomorphism flow we must vary $F$ \textit{explicitly} with time.
	Its total derivative is
	\begin{equation}
		\begin{aligned}
			0\;=\;\frac{dF}{dt}
			\;=\;
			\frac{\partial{}F}{\partial{}t}+\frac{dz}{dt}\frac{\partial{}F}{\partial{}z}
			\;=\;\frac{\partial{}F}{\partial{}t}+\mathcal{A}\frac{\partial{}F}{\partial{}z},
		\end{aligned} \label{EqnsExplFCh}
	\end{equation}
	so we vary $F$ explicitly according to the transport equation $F_t-\mathcal{A}F_z=0$.
	We must also vary $C$ explicitly in order to maintain the conformal gauge.
	Letting $\{\sigma^0,\sigma^1,\sigma^2,\sigma^3\}$ be the orthonormal frame of (\ref{EqnUnitForms}), then of course
	\begin{equation}
		g\;=\;(\sigma^0)^2+(\sigma^1)^2+(\sigma^2)^2+(\sigma^3)^2
	\end{equation}
	and $(\sigma^0)^2=\frac{C}{4F}(dz)^2$, $(\sigma^1)^2=CF(\eta^1)^2$, $(\sigma^2)^2=C(\eta^2)^2$, $(\sigma^2)^2=C(\eta^2)^2$.
	Using $\frac{dF}{dt}=0$ and $L_{\frac{d}{dt}}(dz)^2=2\alpha(dz)^2$, we obtain
	\begin{equation}
		\begin{aligned}
			L_{\frac{d}{dt}}g
			&\;=\;
			\left(2\alpha+\frac{d}{dt}\log\,C\right)(\sigma^0)^2
			+\left(\frac{d}{dt}\log\,C\right)(\sigma^1)^2 \\
			&\quad\quad
			+\left(\frac{d}{dt}\log\,C\right)(\sigma^2)^2
			+\left(\frac{d}{dt}\log\,C\right)(\sigma^3)^2.
		\end{aligned} \label{EqnSecondVarTrFree}
	\end{equation}
	Maintaining the conformal gauge is equivalent to requiring the right-hand side of (\ref{EqnSecondVarTrFree}) be trace-free.
	Thus we vary $C$ by requiring $\frac{d}{dt}\log{}C=-\frac12\alpha$ along the flow, which is the same as evolving $C$ explicitly by $(\log{}C)_t+\mathcal{A}\cdot(\log{}C)_z=-\frac12\alpha$.
	Thus $C$ obeys a transport equation with a source.
	
	\begin{theorem}[The Bach Tensor] \label{ThmBachTensor}
		The Bach tensor of (\ref{EqnMetsCxsKahls}) is
		\begin{equation}
			\begin{aligned}
				Bach
				&\;=\;
				\frac{16}{3C^2}\cdot{}F\cdot\big(\mathcal{L}^-(\mathcal{L}^+(F))-1\big)\cdot\Big(-2(\sigma^1)^2\,+\,(\sigma^2)^2\,+\,(\sigma^3)^2
				\Big) \\
				&\quad\quad
				+\frac{8}{3C^2}\cdot\mathcal{B}(F,F)\cdot\Big(
				-(\sigma^0)^2-(\sigma^1)^2
				+(\sigma^2)^2+(\sigma^3)^2
				\Big).
			\end{aligned} \label{EqnBachDetermination}
		\end{equation}
	\end{theorem}
	\begin{proof}
		We compute the Bach tensor by directly computing the variation of $\int\vert{}W^+\vert^2$.
		First, we use the variation $F\mapsto{}F+sf$ of (\ref{EqnVarFirst}) to compute
		\begin{equation}
			\frac{d}{ds}\int\vert{}W^+\vert^2dVol
			\;=\;\frac{32}{3}\int_{\mathbb{S}^3}\left[\int\mathcal{L}^+(f)(\mathcal{L}^+(F)-1)dz\right]\eta^1\wedge\eta^2\wedge\eta^3.
		\end{equation}
		When integrating against $z$, $\mathcal{L}^-$ is the $L^2$-adjoint of $\mathcal{L}^+$ so therefore
		\begin{equation}
			\frac{d}{ds}\int\vert{}W^+\vert^2dVol
			\;=\;\frac{32}{3}\int{}f\cdot\left(\mathcal{L}^-\left(\mathcal{L}^+(F)\right)-1\right)\,C^{-2}dVol. \label{EqnWNSSVar}
		\end{equation}
		
		Next we use the variation from (\ref{EqnVarZ}).
		Note that $L_{\frac{d}{dt}}dz=\alpha{}dz$ implies $L_{\frac{d}{dt}}dz\wedge\eta^1\wedge\eta^2\wedge\eta^3=\alpha{}dz\wedge\eta^1\wedge\eta^2\wedge\eta^3$.
		Therefore
		\begin{equation}
			\small
			\begin{aligned}
				&
				\frac{d}{dt}\int\vert{}W^+\vert^2dVol
				\;=\;\frac{d}{dt}\frac{16}{3}\int\left(\mathcal{L}^+(F)-1\right)^2\,dz\wedge\eta^1\wedge\eta^2\wedge\eta^3 \\
				&\quad=\frac{16}{3}\int\left[
				2\frac{d}{dt}\left(\mathcal{L}^+(F)\right)\left(\mathcal{L}^+(F)-1\right)
				+\alpha\left(\mathcal{L}^+(F)-1\right)^2
				\right]dz\wedge\eta^1\wedge\eta^2\wedge\eta^3.
			\end{aligned}
		\end{equation}
		Computing $\frac{d}{dt}\mathcal{L}^+(F)$ using $L_{\frac{d}{dt}}\frac{\partial}{\partial{}z}=-\alpha\frac{\partial}{\partial{}z}$ and $\frac{dF}{dt}=0$, we obtain
		\begin{equation}
			\small
			\begin{aligned}
				2\frac{d}{dt}\mathcal{L}^+(F)
				&\;=\;2\frac{d}{dt}\left[\left(-\frac12\frac{\partial}{\partial{}z}+1\right)\left(\left(-\frac{\partial}{\partial{}z}+1\right)F\right)\right] \\
				&\;=\;\alpha\frac{\partial}{\partial{}z}\left(\left(-\frac{\partial}{\partial{}z}+1\right)F\right)
				-2\left(-\frac12\frac{\partial}{\partial{}z}+1\right)\left(\left(-\alpha\frac{\partial}{\partial{}z}\right)F\right) \\
				&\;=\;-\alpha\left(2F_{zz}-3F_z\right)
				-\frac{\partial\alpha}{\partial{z}}F_z.
			\end{aligned}
		\end{equation}
		The variational integral $\frac{d}{dt}\int\vert{}W^+\vert^2dVol$ is therefore
		\begin{equation}
			\small
			\begin{aligned}
				&\frac{d}{dt}\frac{16}{3}\int\left(\mathcal{L}^+(F)-1\right)^2\,dz\wedge\eta^1\wedge\eta^2\wedge\eta^3 \\
				&\hspace{0.25in}=\frac{16}{3}\int\Big(-\alpha\left(2F_{zz}-3F_z\right)
				+\alpha(\mathcal{L}^+(F)-1) \\
				&\hspace{0.875in}
				-F_z\frac{\partial\alpha}{\partial{}z}
				\Big)
				\left(\mathcal{L}^+(F)-1\right)\,dz\wedge\eta^1\wedge\eta^2\wedge\eta^3.
			\end{aligned}
		\end{equation}
		Applying integration by parts to remove the $\frac{\partial\alpha}{\partial{z}}$ factor gives
		\begin{equation}
			\small
			\begin{aligned}
				\frac{d}{dt}\int\vert{}W^+\vert^2\,dVol
				\;=\;\frac{16}{3}\int\alpha\,\mathcal{B}(F,F)\,C^{-2}dVol
			\end{aligned} \label{EqnVarBComp}
		\end{equation}
		where $\mathcal{B}(F,F)$ is the quadratic third order operator of (\ref{EqnBOperatorDef}).
		
		Next we compute the variation of the metric itself.
		Using $(\sigma^0)^2=\frac{C}{4F}(dz)^2$ and $(\sigma^1)^2=CF(\eta^1)^2$ we have $L_{\frac{d}{ds}}g=-\frac{f}{F}(\sigma^0)^2+\frac{f}{F}(\sigma^1)^2$.
		This and (\ref{EqnSecondVarTrFree}) give
		\begin{equation}
			\begin{aligned}
				&L_{\frac{d}{ds}}g
				\;=\;-\frac{f}{F}(\sigma^0)^2
				+\frac{f}{F}(\sigma^1)^2, \\
				&L_{\frac{d}{dt}}g
				\;=\;
				\alpha\frac{3}{2}(\sigma^0)^2
				-\alpha\frac{1}{2}(\sigma^1)^2
				-\alpha\frac{1}{2}(\sigma^2)^2
				-\alpha\frac{1}{2}(\sigma^3)^2.
			\end{aligned} \label{EqnMetrVar}
		\end{equation}
		
		We have now computed $\delta{}g$ and $\delta\int\vert{}W^+\vert^2dVol$.
		The Bach tensor is (implicitly) defined by
		\begin{equation}
			\delta\int\vert{}W^+\vert^2dVol
			\;=\;\int\left<\delta{}g,\,-2Bach\right>\,dVol \label{EqnBachDef}
		\end{equation}
		where we used the fact that the Bach tensor, as it is commonly expressed $Bach_{ij}=W_{sijt}{}^{,st}+\frac12W_{sijt}\Ric^{st}$, is \textit{negative four times}\footnote{Compare (4.77) of \cite{Besse} and ($67^{\prime\prime\prime}$) of \cite{Bach21}.} the $L^2$ gradient of $g\mapsto\int\vert{}W\vert^2dVol$, so negative two times the gradient of $g\mapsto\int\vert{}W^+\vert^2dVol$.
		
		To compute $Bach$, we choose an orthogonal basis for the subspace of the $U(2)$-invariant, trace-free $(0,2)$-forms in which $\delta{}g$ resides:
		\begin{equation}
			I_1=(\sigma^0)^2-(\sigma^1)^2, \qquad
			I_2=(\sigma^0)^2+(\sigma^1)^2-(\sigma^2)^2-(\sigma^3)^2.
		\end{equation}
		These are orthogonal, with $\vert{}I_1\vert^2=2$ and $\vert{}I_2\vert^2=4$.
		We write $Bach=B_1\cdot{}I_1+B_2\cdot{}I_2$	where $B_1$, $B_2$ are functions of $z$.
		The two metric variations (\ref{EqnMetrVar}) are
		\begin{equation}
			\delta_sg\;=\;-\frac{f}{F}\cdot{}I_1
			\qquad\text{and}\qquad
			\delta_tg\;=\;\alpha\cdot{}I_1+\frac{\alpha}{2}\cdot{}I_2.
		\end{equation}
		Substituting each variation into (\ref{EqnBachDef}) gives
		\begin{equation}
			\small
			\begin{aligned}
				&\int{}f\frac{32}{3C^2}\left(\mathcal{L}^-(\mathcal{L}^+(F))-1\right)
				=\int\left<-\frac{f}{F}I_1,\,-B_1I_1-B_2I_2\right>
				=2\int\frac{f}{F}B_1 \\
				&\int\alpha\frac{16}{3C^2}\mathcal{B}(F,F)
				=\int\left<\alpha{}I_1+\frac{\alpha}{2}I_2,\,-B_1I_1-B_2I_2\right>
				=-2\int\alpha\left(B_1+B_2\right)
			\end{aligned}
		\end{equation}
		which must hold for all $f$, $\alpha$.
		Therefore $B_1=\frac{16}{3}F\left(\mathcal{L}^-(\mathcal{L}^+(F))-1\right)$ and $B_2=-\frac83\mathcal{B}(F,F)-\frac{16}{3}F\left(\mathcal{L}^-(\mathcal{L}^+(F))-1\right)$.
		We conclude, as claimed in (\ref{EqnBachDetermination}), that
		\begin{equation}
			\begin{aligned}
				Bach
				&\;=\;
				\frac{16}{3C^2}F\left(\mathcal{L}^-(\mathcal{L}^+(F))-1\right)\left(I_1-I_2\right)
				\;+\;
				\frac{8}{3C^2}\mathcal{B}(F,F)(-I_2).
			\end{aligned}
		\end{equation}
	\end{proof}
	Compare (\ref{EqnBachDetermination}) with (3.3) of \cite{Ot14}; after substituting $C=1$, $F=f^2$ and $dz=2fdt$, these are the same.
	Recalling that $\mathcal{B}$ is a constant of the motion, our expression makes evident that Bach-flat metrics are in fact solutions of a \textit{linear} equation.

	\subsection{Topology: ``nuts,'' ``bolts,'' and asymptotics} \label{SubsecTopology}
	
	Much of our focus is on local analysis in this paper, but here we briefly discuss some global aspects of $U(2)$-invariant metrics.
	Ostensibly the metric (\ref{EqnMetsCxsKahls}) is well defined on $\mathbb{R}\times\mathbb{S}^3$ but topology changes occur if $F$ or $C$ attain $0$ somewhere.
	If $F$ reaches zero, the metric most naturally lives on a quotient $I\times(\mathbb{S}^3/\Gamma)/\sim$ where $\Gamma$ is some discrete subgroup of $SU(2)$, and $\sim$ identifies some 3-sphere to a 2-sphere, via the Hopf map.
	If $F$ or $C$ become infinite there might be a complete (or incomplete) manifold end.
	After reviewing ``nut'' and ``bolt'' topology changes, we discuss ALE, ALF, and asymptotically cusp-like ends.
	
	\subsubsection{Bolts and Nuts} \label{SubSubsecBoltsNuts}
	
	\noindent\begin{figure}[ht]
		\vspace{-0.3in}
		\hspace{.5in}
		\begin{minipage}[c]{0.45\textwidth}
			\caption{
				\it A compact manifold with two bolts, one of positive and one of negative self-intersection.
				At a bolt, the Hopf fiber pinches off to zero while the base $\mathbb{S}^2$ remains finitely-sized.
				\label{FigTwoBolts}
			}
		\end{minipage} \hspace{0.25in}
		\begin{minipage}[c]{0.2\textwidth}
			\includegraphics[scale=0.35]{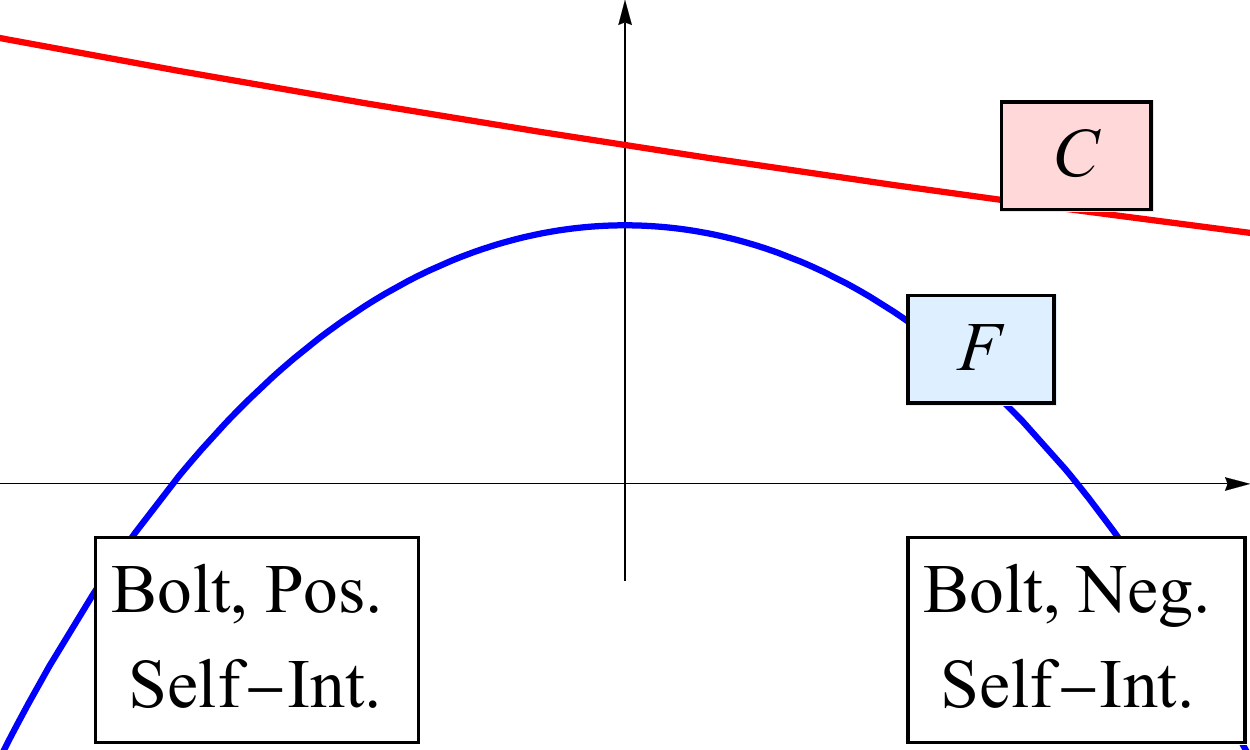}
		\end{minipage} \hfill
	\end{figure}

	The first kind of topology change occurs when the Hopf fiber collapses but the conformal factor remains non-zero, meaning $F$ but not $C$ reaches zero.
	When $F(z_0)=0$,  the locus $z=z_0$ is not a 3-sphere but a 2-sphere, colloquially known as a ``bolt" \cite{GH79} (see also  \cite{Misner63}, \cite{EH78},\cite{LeB88}). 
	
	Recalling the coordinates of Section \ref{SubSecCxStructs}, transversals to the bolt are 2-dimensional submanifolds locally given by $\theta=const$, $\varphi=const$, and the metric is smooth at the bolt provided it is smooth on such transversals.
	The inherited metric on the transversal is $\hat{g}_2=\frac{1}{4F}dz^2\,+\,\frac{F}{4}d\psi^2$, which we write $\hat{g}_2 = dr^2 + (\sqrt{F} d(\frac{1}{2} \psi))^2$ by solving $dr=\frac{1}{\sqrt{4F}}dz$ with $r = 0$ at $z = z_0$. If $\sqrt{F} = k r + O(r^2)$, where $k \neq 0$, then $\psi$ will only have range $(-\frac{2\pi}{k}, \frac{2\pi}{k})$ as we approach the bolt and the metric $\hat{g}_2$ will be conical at $r=0$ with cone angle $2\pi\vert{}k\vert$ (smooth if $k = \pm 1$).
	If $k \in \mathbb{Z}\setminus\{0\}$ however, we can obtain a smooth metric on the quotient $I\times\mathbb{S}^3/\Gamma$ where $\Gamma$ is a cyclic subgroup of order $\vert{}k\vert$ of the Hopf action.
	From $\sqrt{F}=kr+O(r^2)$ we have $k=\frac{d\sqrt{F}}{dr}$, and because $\frac{d}{dr}=2\sqrt{F}\frac{d}{dz}$, $k=\frac{dF}{dz}$.
	We summarize this in the following Proposition.
	
	\begin{proposition}[The ``bolting condition''] \label{LemmaBolting}
		Let $z=z_0$ be a zero of $F(z)$ and not a zero of $C$.
		Assuming
		\begin{equation}
			\left.\frac{dF}{dz}\right\vert_{z=z_0}\;=\;k
		\end{equation}
		where $k\in\mathbb{R}\setminus\{0\}$ then we may identify the locus $\{z=z_0\}$ with a $2$-sphere (a ``bolt'').
		Assuming $k\in\mathbb{Z}\setminus\{0\}$, then taking the $\vert{}k\vert$-to-1 quotient of the $\mathbb{S}^3$ factor, the metric is smooth near $\{z=z_0\}$ and the ``bolt'' is a 2-sphere of self intersection number $k$.
	\end{proposition}
	
	It is possible that two bolts occur, one at $z_0$ and one at $z_1$ where $z_0<z_1$, as in Figure \ref{FigTwoBolts}.
	We certainly must have $\frac{dF}{dz}\ge0$ at $z_0$ and $\frac{dF}{dz}\le0$ at $z_1$ so the bolts, if they are both smooth after resolution, must have self-intersection numbers $k$ and $-k$ where $k\in\mathbb{Z}\setminus\{0\}$.
	With either complex structure $J^+$, $J^-$ this is the same ``odd'' Hirzebruch surface $\Sigma_{2k-1}$; see \cite{KM71}.\\
	
	\noindent\begin{figure}[ht]
		\vspace{-0.3in}
		\hspace{0.75in}
		\begin{minipage}[l]{0.45\textwidth}
			\caption{
				\it Depiction of a ``nut.''
				With $F\approx1$, as $z\rightarrow+\infty$ the 3-sphere is asymptotically round, not squashed.
				Then the conformal factor $e^{-z}$ collapses $z=\infty$ to a point at a finite distance away.
				\label{FigNutEnd}
			} 
		\end{minipage} 
		\hspace{0.125in}
		\begin{minipage}[c]{0.2\textwidth}
			\includegraphics[scale=0.35]{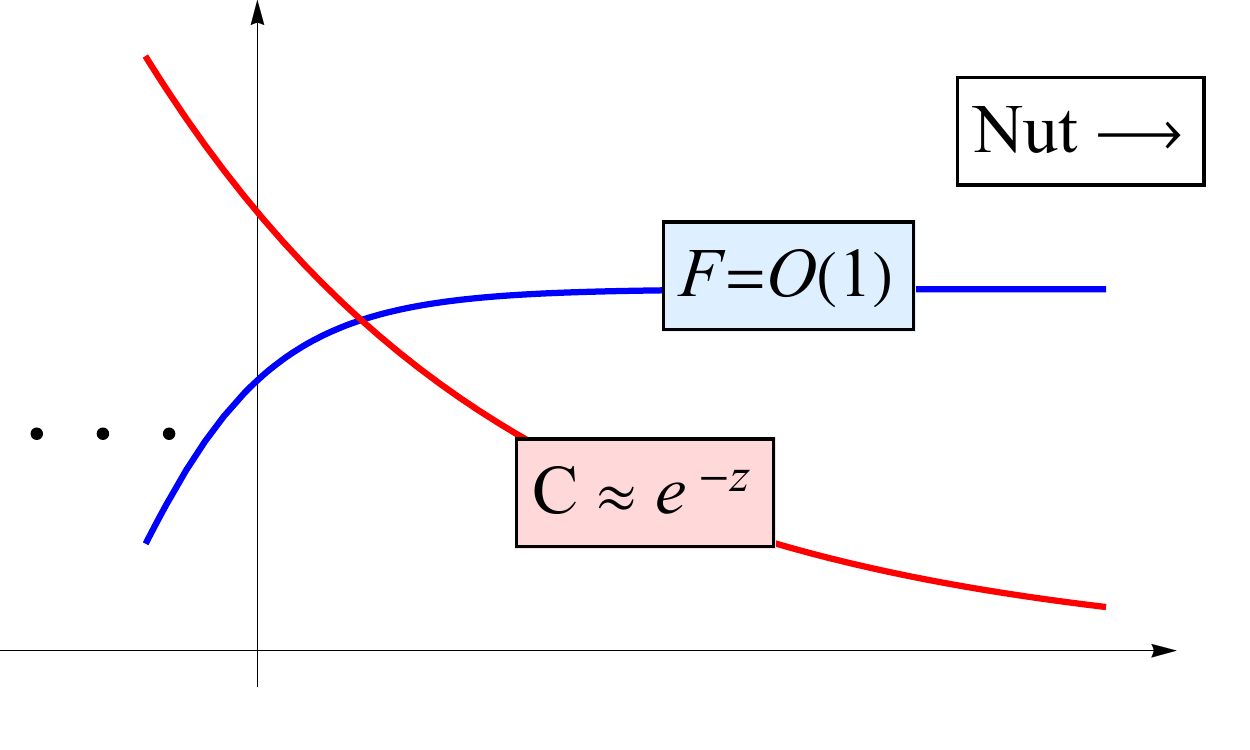}
		\end{minipage}
	\end{figure}
	
	The second topological change, called a ``nut'', occurs when the entire 3-sphere contracts to a point, and the nearby topology is that of a ball in $\mathbb{R}^4$.
	This occurs when $C$ becomes zero but $F$ remains finite.
	When $\omega$ is K\"ahler and $C=C_0e^{-z}$, a nut may occur at $z=+\infty$; this is depicted in Figure \ref{FigNutEnd}.
	When $\omega^-$ is K\"ahler and $C=C_0e^{-z}$ a nut may occur at $z=-\infty$.
	In the non-K\"ahler case, conceivably a nut could form if $C(z)=0$ at some finite $z$, but this produces curvature singularities.
	By considering the change of coordinates $dr=\sqrt{C/4F}\,dz$ near $z=\infty$ and examining the asymptotics of $C$ and $F$, the following proposition is straightforward.
	\begin{proposition}[The Nut condition at $z=\infty$]
		Assume $C=O(e^{-z})$ and $F=1+O(e^{-z})$ as $z\rightarrow\infty$.
		Adding a point at $z=\infty$, this point is a finite distance away and has a neighborhood with bounded curvature and the topology of a ball.
	\end{proposition}
	

	\subsubsection{ALE, ALF, and cusp-like ends} \label{SubSubSecALEALFAndCusp}
	
	\noindent\begin{figure}[ht]
		\hspace{0.75in}
		\begin{minipage}[c]{0.425\textwidth}
			\caption{
				\it Depiction of a K\"ahler ALE end.
				As $z\rightarrow-\infty$, because $F\approx1$ the $\mathbb{S}^2$-factor is approximately round.
				The conformal factor grows like $e^{z}$, and the metric is asymptotically locally Euclidean (ALE).
				\label{FigALEEnd}
			}
		\end{minipage} \hspace{0.1in}
		\begin{minipage}[c]{0.2\textwidth}
			\includegraphics[scale=0.4]{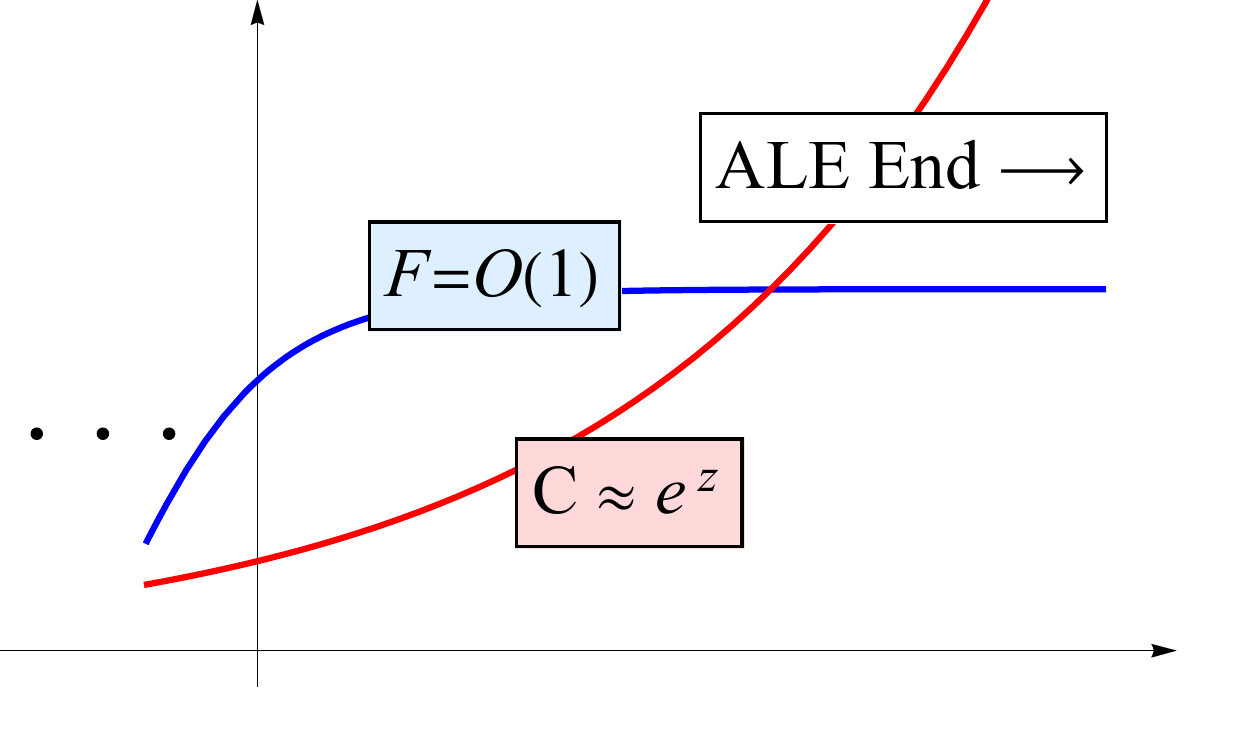}
		\end{minipage} \hfill
	\end{figure}
	
	If $g$ is K\"ahler with respect to $J^+$ so $C=C_0e^{-z}$, an ALE end can occur as $z\rightarrow-\infty$, as depicted in Figure \ref{FigALEEnd}.
	If instead $g$ is K\"ahler with respect to $J^-$ then replacing $z$ by $-z$ Figure \ref{FigALEEnd} is flipped and an ALE end occurs as $z\rightarrow\infty$.
	\begin{proposition}
		Assume $g$ is K\"ahler with respect to $J^+$, so $C=e^{-z}$.
		If $F=1+O(z^{-2})$ as $z\rightarrow-\infty$, the metric is ALE with better-than-quadratically decaying curvature.
	\end{proposition}
	\begin{proof}
		Letting $r$ be the distance function that solves $dr=\frac12\sqrt{C/F}dz=\frac12e^{-\frac12z}(1+O(z^{-2}))dz$, by assumption we have $r=e^{-\frac12z}+O(z^{-1})$.
		Then $C=e^{-z}=r^2+O(r^{-4})$, so the metric is  $g\approx{}dr^2+(r^2+O(r^{-4}))d\sigma_{\mathbb{S}^3}$ as $r\rightarrow\infty$, so it is ALE.
		To check how quickly curvature decays, from Proposition \ref{PropRicForm}
		\begin{equation}
			\begin{aligned}
				\rho
				&\;=\;-2C^{-1}\left(\frac{1}{2}F_{zz}-\frac32F_z+F-1\right)\omega
				+2C^{-1}\left(\frac{1}{2}F_{zz}-\frac12F_{z}-F+1\right)\omega^-
			\end{aligned}
		\end{equation}
		so asymptotically $\rho\approx{}e^{z}O(z^{-2})\omega+e^{z}O(z^{-2})\omega^-=o(r^{-2})$.
		The expressions for $\vert{}W^+\vert$, $\vert{}W^-\vert$ in (\ref{EqnWeylNorm}) give the same decay rates, so we conclude the Riemann tensor decays quadratically: $\vert{}\Riem\vert=o(r^{-2})$.
	\end{proof}

	
	\begin{figure}[ht]
		\includegraphics[scale=0.35]{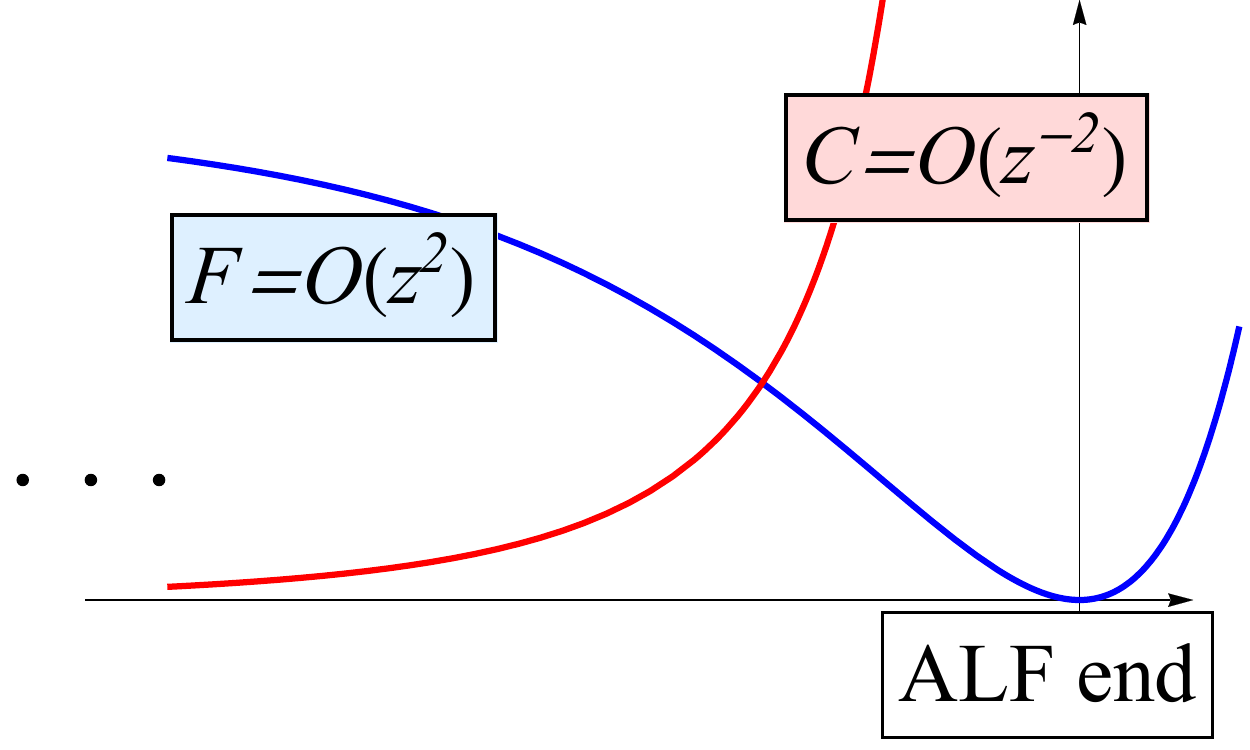}
		\hspace{0.5in}
		\includegraphics[scale=0.35]{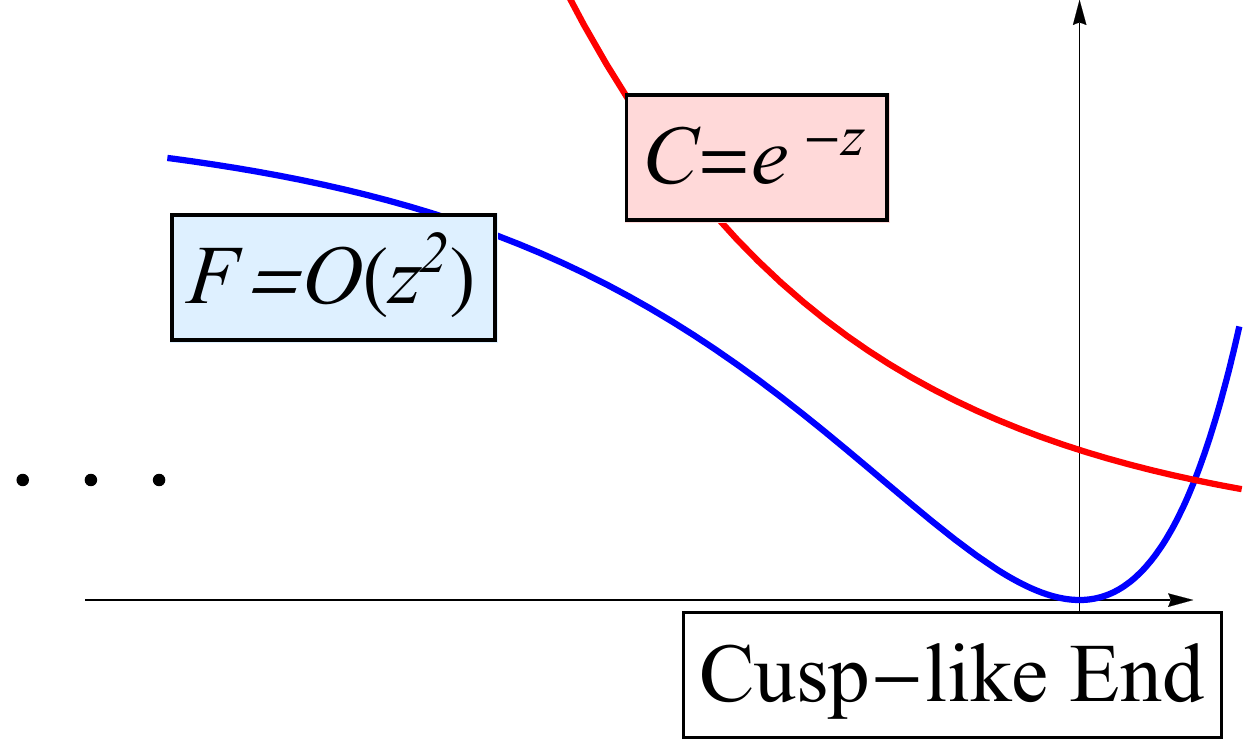}
		\caption{\it An ALF end and a cusp-like end.}
		\label{FigALFandCusp}
	\end{figure}
	
	These are the two kinds of complete ends that can possibly occur at finite coordinate values in the K\"ahler setting.
	The ALF end is familiar from classical examples: these have cubic volume growth, cubic curvature decay, and $\mathbb{R}^3$ tangent cone at infinity.
	See for example \cite{Haw77}, \cite{EGH80}, \cite{CK}, \cite{Etesi}.
	By a ``cusp-like'' end, we mean an end that locally resembles a Riemannian product of a tractrix of revolution (sometimes called a pseudosphere) with a sphere.
	Toward infinity the scalar and Weyl curvatures decrease rapidly, whereas the Ricci curvature approaches a constant bilinear form of signature $(-,-,+,+)$.
	These two kinds of ends are conformal to each other: we will have $C=\frac{e^{z}}{(1-e^{z})^2}$ in the ALF case and $C=e^{-z}$ or $C=e^{z}$ in the cusp-like case.
	$F$ in both cases has a second-order zero at $z=0$.
	See Figure \ref{FigALFandCusp}.
	
	\begin{proposition}
		Assume $F=z^2+O(z^3)$ near $z=0$.
		
		If $C$ remains finite then the manifold forms a complete, cusp-like end near $z=0$.
		Asymptotically the Hopf fiber shrinks to zero and the metric has the local geometry of the product of a psuedosphere times a sphere.
		
		If $C=O(z^{-2})$ then the the metric forms an ALF end near $z=0$.
	\end{proposition}
	\begin{proof}
		The distance function $r$ satisfies $dr=\frac12\sqrt{\frac{C}{F}}dz$ so in the cusp-like case, where $C$ remains finite, then $\sqrt{F}=O(z)$ gives $r\approx\frac12\log(-z)$ near $0$ and indeed the distance to $0$ is infinite so the metric is complete.
		From $\omega\wedge\omega=-C^2dz\wedge\eta^1\wedge\eta^2\wedge\eta^3$, we see the volume is finite.
		Checking the tensors $W^\pm$, with $F=z^2+O(z^3)$ we find that $\mathcal{L}^\pm(F)-1\;=\;O(z)$ and so $\vert{}W^\pm\vert\searrow0$ as $z\rightarrow0$.
		In the K\"ahler case $\rho$ is a multiple of $\omega$ added to a multiple of $\omega^-$.
		The multiple on $\omega$ is also $O(z)$, but the multiple on $\omega^-$, by (\ref{EqnRhoSeparation}), approaches $4C^{-1}$.
		This justifies the assertion that, in the K\"ahler case, the local geometry approaches a $+1$ times a $-1$ curvature surface.
		In the non-K\"ahler case, the usual conformal change formulas for Ricci curvature shows this remains true.
		
		Next we verify that when $C=z^{-2}+O(1)$ near $z=0$, the metric has an ALF end.
		Then $dr=\frac12\sqrt{\frac{C}{F}}dz=\left(\frac12z^{-2}+O(1)\right)dz$ so $r=z^{-1}+O(z)$ near $z=0$.
		To compute volume, we use $C^{\frac32}=O(r^3)$ and $F^{\frac12}=O(z)=O(r^{-1})$, so we have
		\begin{equation}
			dVol
			\;=\;-C^{\frac32}F^{\frac12}dr\wedge{}d\sigma_{\mathbb{S}^3}
			\;\approx\;
			r^2dr\wedge{}d\sigma_{\mathbb{S}^3}. \label{EqnALFVolForm}
		\end{equation}
		Integrating (\ref{EqnALFVolForm}) and noting that $r$ is a distance function, indeed we observe cubic volume growth.
		Next we check curvature decay.
		From (\ref{EqnWeylNorm}) we have $\mathcal{L}^{\pm}(F)-1=O(1)$ so that $\vert{}W^+\vert\approx\frac{32}{3}C^{-2}=O(z^2)=O(r^{-2})$ and similarly for $\vert{}W^-\vert$.
		Inserting $F$, $C$ into the Ricci form $\rho$ from (\ref{EqnRicFormComp}), we see Ricci curvature decays quadratically.
	\end{proof}
	
	We close by remarking that ALE ends and nuts are conformal to each other (by changing between $twoC=e^{-z}$ and $C=e^z$).
	Similarly ALF ends and cusp-like ends are conformal to each other.
	
	\subsubsection{Asymptotically Einsteinian ends, and incomplete ends}
	
	\begin{figure}[ht]
		\centering
		\includegraphics[scale=0.4]{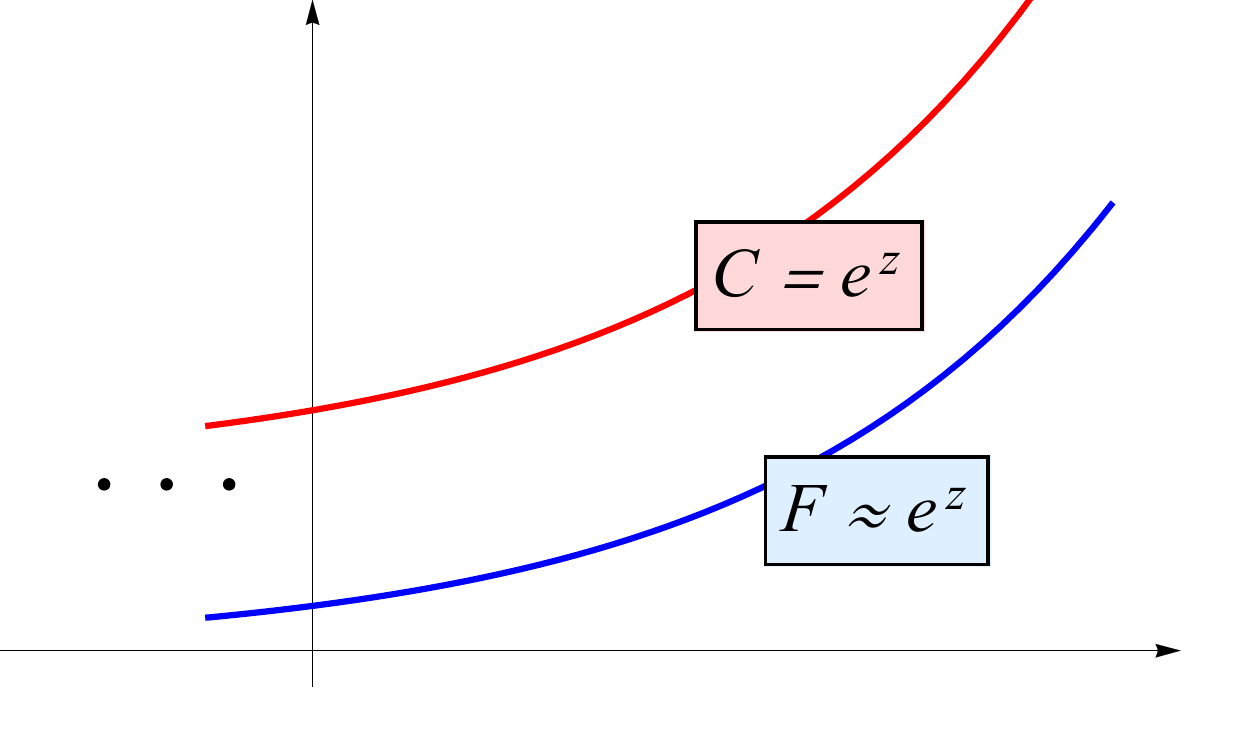}
		\hspace{0.2in}
		\includegraphics[scale=0.4]{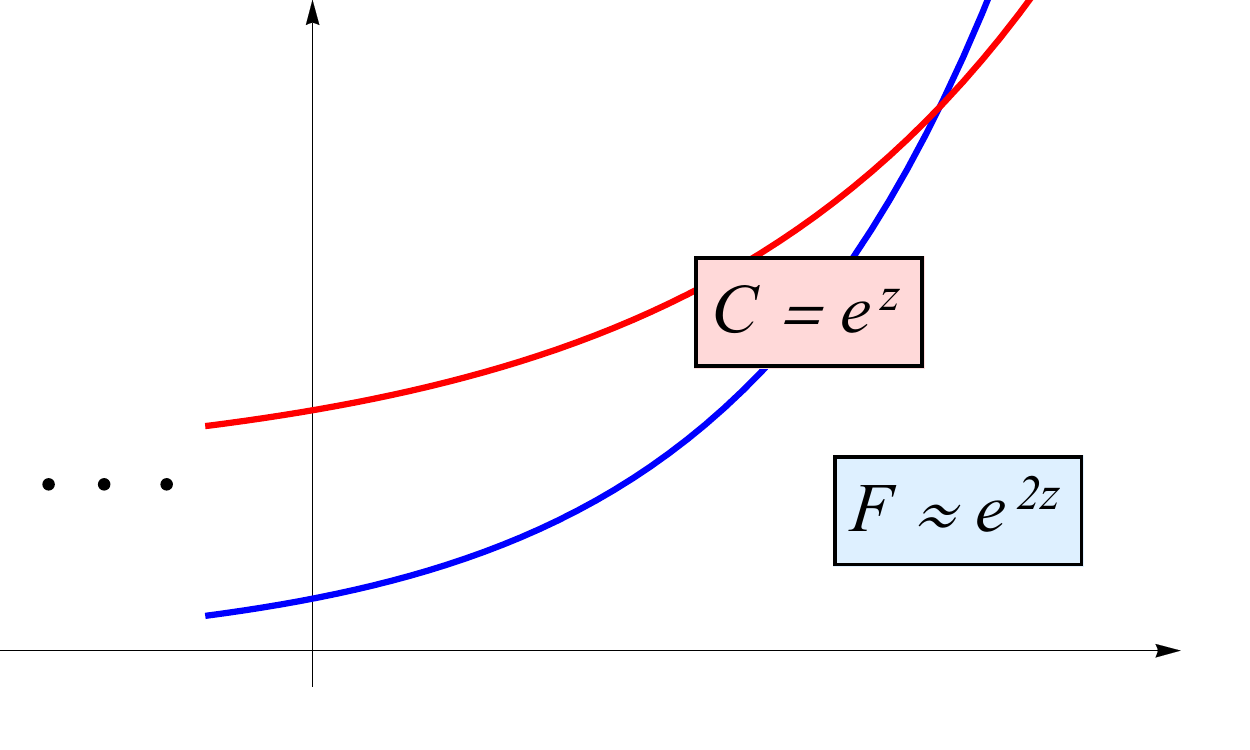}
		\caption{\it Left: an asymptotically Einsteinian end as $z\rightarrow+\infty$.
		Right: an end with a curvature singularity a finite distance away as $z\rightarrow+\infty$.}
		\label{FigEinstAndIncompEnd}
	\end{figure}
	
	Our special metrics all obey $F=1+\frac12C_1e^{-2z}+C_2e^{-z}+C_3e^{z}+\frac12C_4e^{2z}$ and there is the possibility that $F=O(e^{-z})$ or $O(e^{-2z})$ as $z\rightarrow-\infty$.
	In the K\"ahler case $C=C_0e^{-z}$ (or when $J^-$ is the complex structure, $C=C_0e^z$), and $F$ might grow at about the same rate as $C$, or a faster rate.
	When $C=C_0e^{-z}$ and $F=O(e^{-z})$ the end is asymptotically Einsteinian with negative Einstein constant, and when $F=O(e^{-2z})$ the end has a curvature singularity a finite distance away.
	\begin{lemma}
		Assume $g$ is K\"ahler with respect to $J^+$.
		Then as $z\rightarrow\infty$
		\begin{itemize}
			\item if $F=O(e^{-z})$ the metric is asymptotically Einsteinian with negative Einstein constant.
			\item if $F=O(e^{-2z})$ the metric is incomplete and $s\searrow-\infty$ near this end.
		\end{itemize}
		The same holds if $g$ is K\"ahler with respect to $J^-$, after replacing $z$ by $-z$.
	\end{lemma}
	\begin{proof}
		Proposition \ref{PropRicForm} states
		\begin{equation}
			\begin{aligned}
				\rho
				&=-\frac{2}{C}\left(\frac12F_{zz}-\frac32F_z+F-1\right)\omega^+
				+\frac{2}{C}\left(\frac12F_{zz}-\frac12F_{z}-F+1\right)\omega^-.
			\end{aligned} \label{EqnBreakupRho}
		\end{equation}
		If $F$ blows up like $C_3e^{-z}$ as $z\searrow-\infty$, then (\ref{EqnBreakupRho}) gives
		\begin{equation}
			\rho\;\approx\;-2C_0^{-1}e^{z}(3C_3e^{-z}-1)\omega^++2C_0^{-1}e^{z}\omega^-
		\end{equation}
		and asymptotically we obtain $\rho\approx-6C_3C_0^{-1}\omega^+$, so $g$ is indeed asymptotically Einstein with negative Einstein constant.
		
		Finally if $F$ blows up like $e^{2z}$ then the distance function satisfies $dr=\frac12\sqrt{C/F}\,dz\approx{}e^{-\frac12z}dz$ and so the distance is finite as $z\searrow-\infty$.
		Using equation (\ref{EqnBreakupRho}) again, we see that $\rho$ blows up like $e^{-z}$, so curvature also blows up at a finite distance.
	\end{proof}

	\section{Special Metrics}
	
	We use the computations of Section \ref{SubsecCurvQuants} to determine what conditions are needed to make a $U(2)$-invariant metric special or canonical.
	
	\subsection{Scalar Curvature}
	
	From (\ref{EqnScalGeneral}) of Proposition \ref{PropRicci}, specifying scalar curvature is equivalent to
	\begin{equation}
		sC^{\frac32}
		+4C^{\frac12}\left(\frac{\partial^2F}{\partial{z}^2}+\frac12F-2\right)
		+24\frac{\partial}{\partial{}z}\left(F\frac{\partial}{\partial{}z}C^{\frac12}\right)\;=\;0. \label{EqnScalSecond}
	\end{equation}
	for given $s=s(z)$.
	This relation is linear in $F$, and degenerates only when $F$ or $C$ reaches zero.
	It is underdetermined.
	Imposing, for example, the K\"ahler condition creates a critically determined equation.

	\subsection{Extremal K\"ahler metrics}
	
	A K\"ahler metric is extremal if the functional $g\mapsto\int{}s^2dVol$ is stable under those perturbations of $g$ that preserve the K\"ahler class.
	From \cite{Cal82} the Euler-Lagrange equations are that the gradient $\nabla{}s$ is a holomorphic vector field, but there are several ways to assess whether (\ref{EqnMetrInZ}) is extremal.
	We use the condition that a K\"ahler metric is extremal if and only if $J\nabla{}s$ is Killing.
	
	\begin{proposition}[The extremal condition] \label{PropExtremal}
		The metric (\ref{EqnMetsCxsKahls}) with complex structure $J^+$ is extremal K\"ahler if and only if $C=C_0e^{-z}$ and $\mathcal{L}^-(\mathcal{L}^+(F))=1$, which is
		\begin{equation}
			F\;=\;1
			+\frac12C_1e^{-2z}
			+C_2e^{-z}
			+C_3e^{z}
			+\frac12C_4e^{2z}. \label{EqnFforExtremal}
		\end{equation}
		Scalar curvature is $s=-\frac{24}{C_0}(C_1e^{-z}+C_2)$, and $g$ is CSC-K if also $C_1=0$.
		
		Likewise, the metric with complex structure $J^-$ is extremal K\"ahler if and only if $C=C_0e^{z}$ and again $\mathcal{L}^-(\mathcal{L}^+(F))=1$.
		It has scalar curvature $s=-\frac{24}{C_0}(C_3+C_4e^{z})$, and the metric is CSC-K when $C_4=0$.
	\end{proposition}
	\begin{proof}
		From (\ref{EqnMetsCxsKahls}) and (\ref{EqnFrTransLeft}), we have $\nabla{}z=4\frac{F}{C}\frac{\partial}{\partial{}z}=\frac{4}{C}J\frac{\partial}{\partial\psi}$.
		Because the coordinate field $\frac{\partial}{\partial\psi}$ is itself a Killing field and because $s=s(z)$ is a function of $z$ alone, for $J\nabla{}s$ to be Killing it must be proportional $\frac{\partial}{\partial\psi}$.
		Thus the extremal condition is
		$
		\nabla{}s
		=-4\alpha{}J\frac{\partial}{\partial\psi}
		=-\alpha{}e^{z}\nabla{}z
		=\nabla\left(\alpha{}e^{-z}\right)
		$
		for any constant $\alpha$.
		Therefore $s=\alpha{}e^{-z}+\beta$ where $\beta$ is another constant.
		Using $s=-8C_0^{-1}e^{z}(\mathcal{L}^+(F)-1)$ from (\ref{EqnScalCurvComp}) we obtain
		\begin{equation}
			-8C_0^{-1}e^{z}\left(\frac{1}{2}\frac{\partial^2F}{\partial{}z^2}
			-\frac32\frac{\partial{}F}{\partial{}z}
			+F-1\right)\;=\;\alpha{}e^{-z}+\beta.
		\end{equation}
		The general solution is (\ref{EqnFforExtremal}), after setting $C_1=-\frac{1}{24}\alpha{}C_0$ and $C_2=-\frac{1}{24}\beta{}C_0$.
		
		For $J^-$ in place of $J^+$, reverse the sign on $z$ in all computations.
	\end{proof}

	\subsection{Ricci curvature and Einstein metrics}
	
	By (\ref{EqnTfRic}), $\cRic=0$ and the metric is Einstein if and only if
	\begin{equation}
		\begin{aligned}
			&\frac{\partial^2}{\partial{}z^2}C^{-\frac12}=\frac14C^{-\frac12}
			\;\;\text{and}\;\;
			C^{\frac12}\frac{\partial}{\partial{}z}\left(F\frac{\partial}{\partial{}z}C^{-\frac12}\right)=\left(\frac12\frac{\partial^2F}{\partial{}z^2}-\frac34F+1\right).
		\end{aligned} \label{EqnsRicFlatSimpler}
	\end{equation}
	This is critically determined and partly decoupled.
	It is $4^{th}$ order in total so we will have a 4-parameter solution space.
	The general solution is
	\begin{equation}
		\begin{aligned}
			&F\;=\;1+\frac12C_1e^{-2z}+C_2e^{-z}+C_3e^z+\frac12C_4e^{2z}, \;\;
			C\;=\;\frac{e^{-z}}{\left(C_5+C_6e^{-z}\right)^2}, \\
			&\text{where}\quad
			C_1C_5-C_2C_6=0,
			\quad \text{and} \quad
			C_3C_5-C_4C_6=0.
		\end{aligned} \label{EqnEinsteinSol}
	\end{equation}
	With six constants and two algebraic relations we have the expected four-parameter family of solutions.
	(We remark that the Einstein metrics of Proposition \ref{PropIntegrability} have this form.)
	The two algebraic relations can be expressed in vector form: the vectors $(C_1,C_2)$, $(C_3,C_4)$, and $(C_5,C_6)$ are proportional to each other.
	Because $(C_1,C_2)$ and $(C_3,C_4)$ are proportional, we have $C_1C_4-C_2C_3=0$ so we recover the fact that Einstein metrics are Bach-flat; see (\ref{EqnThreeParamBachFlat}) below.
	By Lemma \ref{LemmaClosedJ} the metric is K\"ahler when $C_6=0$ (for $J^+$) or $C_5=0$ (for $J^-$).
	
	To be Ricci-flat, $C$ and $F$ require, in addition to (\ref{EqnsRicFlatSimpler}), that $s=0$.
	This third relation appears to make the overall system overdetermined, but it does not, for the reason that $s$ is a first integral for the system (\ref{EqnsRicFlatSimpler}) so only contributes an algebraic relation.
	From (\ref{EqnScalSecond}),
	\begin{equation}
		s\;=\;-24(C_2C_5^2-2C_5C_6+C_3C_6^2). \label{EqnScalarLate}
	\end{equation}
	\begin{proposition}[The Einstein conditions] \label{PropEinstCondition}
		The metric (\ref{EqnMetrInZ}) is Einstein if and only if
		\begin{equation}
			\begin{aligned}
				&F\;=\;1+\frac12C_1e^{-2z}+C_2e^{-z}+C_3e^z+\frac12C_4e^{2z}, \quad
				C\;=\;\frac{e^{-z}}{\left(C_5+C_6e^{-z}\right)^2}, \\
				&C_1C_5-C_2C_6=0, \quad \text{and} \quad C_3C_5-C_4C_6=0.
			\end{aligned} \label{EqnThmEinstFC}
		\end{equation}
		Its scalar curvature is the constant $s=-24(C_2C_5^2-2C_5C_6+C_3C_6^2)$.
		
		Up to homothety, there is a 2-dimensional family of Einstein metrics.
		Up to homothety, there is a 1-dimensional family of Ricci-flat metrics, a 1-dimensional family of KE metrics with respect to $J^+$, and a 1-dimensional family of KE metrics with respect to $J^-$.
		Up to homothety and biholomorphism, there are exactly five Ricci-flat K\"ahler metrics.
	\end{proposition}
	\begin{proof}
		We have proven everything except the final assertion, that exactly five metrics of the form (\ref{EqnMetrInZ}) are Ricci-flat K\"ahler, up to homothety.
		We prove this regardless of the complex structure, whether one of the structures considered here or not.
		A $U(2)$-invariant metric is Einstein if and only if it has the form (\ref{EqnThmEinstFC}).
		By Derdzinski's theorem \cite{Derd}, if a scalar-flat metric is K\"ahler---regardless of the complex structure---then it is half-conformally flat.
		In particular either $C_1=C_2=0$ or $C_3=C_4=0$.
		
		So assume $C_3=C_4=0$; the case $C_1=C_2=0$ is identical under the isomorphism $z\mapsto-z$.
		We have four remaining variables $C_1,C_2,C_5,C_6$ and two relations: $C_1C_5-C_2C_6=0$ from (\ref{EqnEinsteinSol}) and $C_2C_5^2-2C_5C_6=0$ from (\ref{EqnScalarLate}).
		If in addition to $C_3=C_4=0$ we have both $C_1=C_2=0$ then either $C_5=0$ or else $C_6=0$ and in either case we have the flat metric: $F=1$ and $C=C_0e^{\pm{}z}$.
		
		Suppose $C_1=0$ but $C_2\ne0$; then the two relations force $C_5=C_6=0$, an impossibility.
		Suppose $C_1\ne0$ but $C_2=0$; then the relations force $C_6=0$ so
		\begin{equation}
			F\;=\;1+\frac12C_1e^{-2z}, \quad\quad C\;=\;\frac{1}{C_5^{2}}e^{-z}
		\end{equation}
		which is K\"ahler with respect to $J^+$.
		Up to homothety, there are exactly two such metrics: the first of these is given by $F=1-e^{-2z}$, $C=e^{-z}$, which is the Eguchi-Hanson metric (see the table in Appendix \ref{SecClassic}), and the second is given by
		\begin{equation}
			F=1+e^{-2z}, \quad C=e^{-z}
			\label{EqnSuperEguchiHanson}
		\end{equation}
		which is incomplete and has a curvature singularity at $z=-\infty$.
		
		Lastly it is possible that neither $C_1$ nor $C_2$ are zero.
		The two relations now give $\frac{C_6}{C_5}=\frac{C_1}{C_2}$ and $\frac{C_6}{C_5}=\frac{C_2}{2}$, so $C_1=\frac12C_2^2$.
		Therefore the metric is
		\begin{equation}
			F=1+\frac14C_2^2e^{-2z}+C_2e^{-z}=\left(1+\frac12C_2e^{-z}\right)^2, \;\; C=\frac{C_5^2e^{-z}}{\left(1+\frac12C_2e^{-z}\right)^2}. \label{EqnConclTaubNut}
		\end{equation}
		Under the isomorphism $z\mapsto-z$ this is the K\"ahler metric of Proposition \ref{PropIntegrability} which is K\"ahler with respect to the complex structure $I^-$; therefore the metric (\ref{EqnConclTaubNut}) is K\"ahler with respect to the complex structure $I^+$.
		As in Proposition \ref{EqnTwoRFEMetrics}, we obtain exactly two metrics: one where $C_2<0$ (which is isomorphic to the Taub-NUT) and one where $C_2>0$ (which has a curvature singularity).
	\end{proof}

	\subsection{Half-conformally flat, half-harmonic, and Bach-flat metrics}
	
	\begin{proposition}[Half-conformally flat metrics] \label{PropHCFmetrics}
		The metric (\ref{EqnMetrInZ}) is half-conformally flat with $W^\pm=0$ if and only if $\mathcal{L}^\pm(F)-1=0$, meaning
		\begin{equation}
			F\;=\;1+C_3e^{z}+\frac12C_4e^{2z}\quad\text{or}\quad
			F\;=\;1+\frac12C_1e^{-2z}+C_2e^{-z},
		\end{equation}
		respectively.
		Up to homothety, each case constitutes a 1-parameter family of such metrics, each a subspace of the 2-parameter family of Bach-flat metrics.
		
		In the case $g$ is K\"ahler with respect to $J^+$ so $C=C_0e^{-z}$, then $W^+=0$ implies $s=0$, and $W^-=0$ implies $s=-\frac{24}{C_0}(C_1e^{-z}+C_2).$
	\end{proposition}
	
	The half-harmonic condition $\delta{}W^+=0$ (or $\delta{}W^-=0$) is underdetermined, and requires an additional condition to be critically determined.
	Three possibilities suggest themselves: $s=const$, the K\"ahler condition, and both $\delta{}W^\pm=0$.
	
	\begin{proposition}[Half-Harmonic metrics]
		The metric (\ref{EqnMetrInZ}) has $\delta{}W^+=0$ if and only if a constant $k_1$ exists so $e^{\frac32z}\left(\mathcal{L}^+(F)-1\right)C=k_1$, and $\delta{}W^-=0$ if and only if $e^{-\frac32z}\left(\mathcal{L}^-(F)-1\right)C=k_2$, some $k_2\in\mathbb{R}$.
		
		Assume (\ref{EqnMetsCxsKahls}) is K\"ahler with respect to $J^+$, meaning $C=C_0e^{-z}$.
		Then
		\begin{itemize}
			\item[{\it{a}})] $\delta{}W^+=0$ if and only if $F=1+C_2e^{-z}+C_3e^z+\frac12C_4e^{2z}$.
			In particular scalar curvature $s=-24\frac{C_2}{C_0}$ is constant.
			\item[{\it{b}})] $\delta{}W^-=0$ if and only if $F=1+\frac12C_1e^{-2z}+C_2e^{-z}+\frac12C_4e^{2z}$.
			In particular the metric is extremal and $s=-24\frac{1}{C_0}(C_1e^{-z}+C_2)$.
		\end{itemize}
	\end{proposition}
	\begin{proof}
		The assertion for $\delta{}W^+=0$ follows from Proposition \ref{PropDelWpm} using $C=C_0e^{-z}$, $e^{\frac32z}(\mathcal{L}^+(F)-1)\sqrt{C}=k_1$ and finding the general solution.
		In the K\"ahler case, ({\it{a}}) and ({\it{b}}) follow from Proposition \ref{PropExtremal}.
	\end{proof}
	
	\vspace{0.1in}
	
	The harmonic-Weyl condition $\delta{}W=0$ is the two conditions $\delta{}W^\pm=0$.
	This is critically determined in the $U(2)$-invariant case, which may be surprising because it is underdetermined on a generic 4-manifold and requires an additional condition, usually on scalar curvature \cite{DerdHarm}, to make it critically determined.
	In any case an Einstein metric has $\delta{}W=0$, while in the $U(2)$-invariant case the converse is also true.
	
	\begin{proposition}[Harmonic curvature]
		The metric (\ref{EqnMetsCxsKahls}) has $\delta{}W=0$ if and only if $g$ is Einstein.
	\end{proposition}
	\begin{proof}
		Because $\delta{}W^+\in{}T^*M\otimes\bigwedge^+$ and $\delta{}W^-\in{}T^*M\otimes\bigwedge^-$, we have $\delta{}W=0$ if and only if $\delta{}W^+$ and $\delta{}W^-$ are both zero.
		Then by Lemma \ref{PropDelWpm} constants $k_1$, $k_2$ exist so
		\begin{equation}
			\begin{aligned}
				e^{\frac32z}(\mathcal{L}^+(F)-1)\sqrt{C}=k_1 \quad\text{and}\quad
				e^{-\frac32z}(\mathcal{L}^-(F)-1)\sqrt{C}=k_2.
			\end{aligned} \label{EqnHarmWeylConds}
		\end{equation}
		Eliminating $C$, we obtain $k_2e^{\frac32z}(\mathcal{L}^+(F)-1)=k_1e^{-\frac32z}(\mathcal{L}^-(F)-1)$ which has general solution
		\begin{equation}
			F\;=\;1
			+k_1\left(\frac12C_1e^{-2z}
			+C_2e^{-z}\right)
			+k_2\left(C_1e^z
			+\frac12C_2e^{2z}\right). \label{EqnFHarmWeyl}
		\end{equation}
		Using either equation in (\ref{EqnHarmWeylConds}), $C=\frac{C_0e^{-z}}{\left(C_2+C_1e^{-z}\right)^2}.$
		From Proposition \ref{PropEinstCondition}, the metric is Einstein.
	\end{proof}
	
	Next we consider the case of Bach-flat metrics.
	From Proposition \ref{ThmBachTensor} this requires $F$ solve the fourth order linear equation $\mathcal{L}^-(\mathcal{L}^+(F))-1=0$ and the third order non-linear equation $\mathcal{B}(F,F)=0$.
	This seems to be overdetermined, but by (\ref{EqnBConstMotLLF}) the two equations are not independent.
	\begin{lemma} \label{LemmaBSols}
		If $F$ solves $\mathcal{L}^+(\mathcal{L}^-(F))-1$ then $\mathcal{B}(F,F)=const$.
		If $F$ solves $\mathcal{B}(F,F)=0$, then $\mathcal{L}^+(\mathcal{L}^-(F))-1=0.$
		$F$ solves both equations if and only if
		\begin{equation}
			F=1+\frac12C_1e^{-2z}+C_2e^{-z}+C_3e^z+\frac12C_4e^{2z}
			\;\;\text{and}\;\;
			C_1C_4-C_2C_3=0. \label{EqnThreeParamBachFlat}
		\end{equation}
	\end{lemma}
	\begin{proof}
		A tedious but straightforward computation shows
		\begin{equation}
			\frac{\partial}{\partial{}z}\mathcal{B}(F,F)
			\;=\;2\frac{\partial{}F}{\partial{}z}
			\cdot\left(\mathcal{L}^+(\mathcal{L}^-(F))-1\right). \label{EqnBConstMotLLF}
		\end{equation}
		Therefore $\mathcal{B}(F,F)$ is indeed constant on solutions of $\mathcal{L}^+(\mathcal{L}^-(F))-1=0$, as claimed.
		Further, $\mathcal{B}(F,F)=0$ implies that either $F=const$ or $\mathcal{L}^+(\mathcal{L}^-(F))=1$.
		Direct computation shows the only constant that satisfies $\mathcal{B}(F,F)=0$ is $F=1$, which indeed solves $\mathcal{L}^+(\mathcal{L}^-(F))-1=0$.
		We conclude that $\mathcal{B}(F,F)=0$ implies $\mathcal{L}^+(\mathcal{L}^-(F))-1=0$, as claimed.
		
		The general solution of $\mathcal{L}^+(\mathcal{L}^-(F))=1$ is $F=1+\frac12C_1e^{-2z}+C_2e^{-z}+C_3e^z+\frac12C_4e^{2z}$, and in this case direct computation shows that $\mathcal{B}(F,F)=3(C_2C_3-C_1C_4)$.
		Therefore the general solution of $\mathcal{L}^+(\mathcal{L}^-(F))=1$, $\mathcal{B}(F,F)=0$ is the three parameter family of (\ref{EqnThreeParamBachFlat}).
	\end{proof}
	
	\begin{proposition}
		The metric (\ref{EqnMetsCxsKahls}) is Bach-flat if and only if
		\begin{equation}
			F=1+\frac12C_1e^{-2z}+C_2e^{-z}+C_3e^z+\frac12C_4e^{2z}
			\quad\text{and}\quad
			C_1C_4-C_2C_3=0.
		\end{equation}
		In particular $g$ is Bach-flat if and only if it is conformally Einstein.
		Up to choice of conformal factor and translation in $z$, there is precisely a 2-parameter family of Bach-flat metrics.
	\end{proposition}
	\begin{proof}
		The metric $g$ is Bach-flat if and only if $\mathcal{L}^+(\mathcal{L}^-(F))-1=0$ and $\mathcal{B}(F,F)=0$.
		From Lemma \ref{LemmaBSols}, this holds if and only if $F=1+\frac12C_1e^{-2z}+C_2e^{-z}+C_3e^z+\frac12C_4e^{2z}$ and $C_1C_4-C_2C_3=0$, giving a 3-parameter family of solutions.
		Factoring out by translation in $z$, this is a 2-parameter family, as claimed.
		To see that any Bach-flat metric is conformal to an Einstein metric, simply let $C$ be a conformal factor from Proposition \ref{PropEinstCondition}.
	\end{proof}

	\subsection{$B^t$-flat metrics}
	
	Lastly we discuss the $B^t$-flat metrics of \cite{GV16}.
	These are extremizers of the quadratic functional $B^t(g)=\int\vert{}W\vert^2+t\int{}s^2$, and when $t=\infty$ we take $B^\infty=\int{}s^2$.
	The Euler-Lagrange equations of this functional \cite{GV16} are
	\begin{equation*}
		-4Bach\,+\,t\,\mathcal{C}\;=\;0
	\end{equation*}
	where $\mathcal{C}=2\left(\nabla^2s-(\triangle{}s)g-s\cRic\right)$ and for the Bach tensor we use the classic expression $Bach_{ij}=W_{sijt,st}+\frac12W_{sijt}\Ric_{st}$.
	Using $Tr(-4Bach+t\mathcal{C})=0$, when $t\ne0$ a $B^t$-flat metric automatically has $\triangle{}s=0$.
	This means the $B^t$-flat equations are the two equations $2Bach+t(s\cRic-\nabla^2s)=0$ and $\triangle{}s=0.$
	We re-express these as an ODE system.
	
	\begin{lemma}[The unreduced $B^t$-flat equations] \label{LemmaUnReduced}
		In the metric (\ref{EqnMetsCxsKahls}) the $B^t$-flat equations $\triangle{}s=0$, $2Bach+t(s\cRic-\nabla^2s)=0$ are equivalent to
		\begin{equation}
			\begin{aligned}
				\frac{\partial}{\partial{}z}\left(
				CF\frac{\partial{}s}{\partial{}z}\right)=0, \quad
				\mathcal{F}_1(F,C)\;=\;0, \quad
				\mathcal{F}_2(F,C)\;=\;0, \quad
				\mathcal{T}(F,C)\;=\;0
			\end{aligned} \label{EqnFourBtFlatEqns}
		\end{equation}
		where $\mathcal{F}_1$, $\mathcal{F}_2$ and $\mathcal{T}$ are the operators
		\begin{equation}
			\small
			\begin{aligned}
				&\mathcal{F}_1(F,C)
				=24\frac{\partial}{\partial{}z}\left(F\frac{\partial}{\partial{}z}C^\frac12\right)
				+4C^{\frac12}\left(\frac{\partial^2F}{\partial{}z^2}+\frac12F-2\right)+s\,C^{\frac32}\\
				&\mathcal{F}_2(F,C)
				=\frac83\left(\mathcal{L}^+(\mathcal{L}^-(F))-1\right)
				+t\,s\,C^{\frac32}\left(\frac{\partial^2}{\partial{}z^2}C^{-\frac12}-\frac14C^{-\frac12}\right)\\
				&\hspace{0.6in}
				+\frac{t}{2}\frac{C}{F}\frac{\partial{}F}{\partial{}z}\frac{\partial{}s}{\partial{}z}
				+t\frac{\partial{}C}{\partial{}z}\frac{\partial{}s}{\partial{}z} \\
				&\mathcal{T}(F,C)
				\;=\;16\,\mathcal{B}(F,F)
				-18tF\frac{\partial{}C}{\partial{}z}\frac{\partial{}s}{\partial{}z}
				-6tC\frac{\partial{}F}{\partial{}z}\frac{\partial{}s}{\partial{}z} \\
				&\quad
				-\frac34tsC^{-1}\left(
				C^2(-16+4F+Cs)
				+12F\left(\frac{\partial{}C}{\partial{}z}\right)^2
				+8C\frac{\partial{}C}{\partial{}z}\frac{\partial{}F}{\partial{}z}
				\right)
			\end{aligned}
		\end{equation}
		and $\mathcal{B}$ is the operator from (\ref{EqnBOperatorDef}).
	\end{lemma}
	\begin{proof}
		In coordinates, $\triangle=\frac{1}{\sqrt{det\,g}}\frac{\partial}{\partial{}x^i}\left(\sqrt{det\,g}\,g^{ij}\frac{\partial}{\partial{}x^j}\right)$.
		Using $(Z,\tau,x,y)$-coordinates of (\ref{EqnLeBrunForm}) we have $det\,g=\frac{1}{16\cosh^2(x)}C^2$ and $g^{11}=4FC$.
		Because $s=s(Z)$ is a function of $Z$ alone, then $0=\triangle{}s$ is
		\begin{equation}
			\begin{aligned}
				0&=\frac{4\cosh^2(x)}{C}\frac{\partial}{\partial{}Z}\left(\frac{C}{4\cosh^2(x)}4FC\frac{\partial{}s}{\partial{}Z}\right)
				\;=\;\frac{4}{C}\frac{\partial}{\partial{}Z}\left(FC^2\frac{\partial{}s}{\partial{}Z}\right)
			\end{aligned}
		\end{equation}
		The coordinate transition from $z$ to $Z$ of (\ref{EqnLeBrunForm}) is $C\frac{\partial}{\partial{}Z}=\frac{\partial}{\partial{}z}$, so we obtain the first equation of (\ref{EqnFourBtFlatEqns}).
		The second equation $\mathcal{F}_1(F,C)=0$ is precisely the scalar curvature equation (\ref{EqnScalSecond}).
		With $\triangle{}s=0$ the Hessian $\nabla^2s$ is trace-free.
		One computes that
		\begin{equation}
			\begin{aligned}
				\nabla^2s
				&\;=\;
				-2C^{-4}\frac{\partial{}s}{\partial{}z}\frac{\partial(FC^3)}{\partial{z}}\,(\sigma^0)^2
				+2C^{-2}\frac{\partial{}s}{\partial{}z}\frac{\partial(FC)}{\partial{z}}\,(\sigma^1)^2
				+2FC^{-2}\frac{\partial{}s}{\partial{}z}\frac{\partial{}C}{\partial{z}}\,\Big((\sigma^2)^2+(\sigma^3)^2\Big).
			\end{aligned} \label{EqnHessR}
		\end{equation}
		
		For the third and fourth equations we use (\ref{EqnTfRic}), (\ref{EqnBachDetermination}), and (\ref{EqnHessR}).
		We expect precisely two additional relations, due to the fact that each of the tensors $Bach$, $\cRic$, and $\nabla^2s$ have four non-zero components, but also the two algebraic relations of being trace-free, and having identical $(3,3)$ and $(4,4)$ entries.
		We take one relation from $2(Bach_{00}+Bach_{22})+t(s\cRic{}_{00}+s\cRic{}_{22}-s_{,00}-s_{,22})=0$.
		Using (\ref{EqnTfRicPre}), (\ref{EqnBachPre}), and (\ref{EqnHessR}), this is
		\begin{equation}
			\small
			\begin{aligned}
				\frac83\left(\mathcal{L}^-(\mathcal{L}^+(F))-1\right)
				+tsC^{\frac32}\left(\frac{\partial^2C^{-\frac{1}{2}}}{\partial{}z^2}-\frac{1}{4}C^{-\frac{1}{2}}\right)
				+\frac{t}{2}\frac{C}{F}\frac{\partial{}F}{\partial{}z}\frac{\partial{}s}{\partial{}z}
				+t\frac{\partial{}C}{\partial{}z}\frac{\partial{}s}{\partial{}z}
				=0
			\end{aligned}
		\end{equation}
		which is $\mathcal{F}_2(C,F)=0$.
		And we take another relation from $2Bach_{00}+t(s\cRic{}_{00}-s_{,00})=0$, which is
		\begin{equation}
			\small
			\begin{aligned}
				&0\;=\;16\mathcal{B}(F,F)
				-18tF\frac{\partial{}C}{\partial{}z}\frac{\partial{}s}{\partial{}z}
				-6tC\frac{\partial{}F}{\partial{}z}\frac{\partial{}s}{\partial{}z} \\
				&\hspace{0.0125in}
				-\frac34tsC^{-1}\left(
				C^2(-16+4F+sC)
				+12F\left(\frac{\partial{}C}{\partial{}z}\right)^2
				+8C\frac{\partial{}C}{\partial{}z}\frac{\partial{}F}{\partial{}z}
				\right) \\
				&\hspace{0.0125in}
				+\frac34tsC^{\frac12}
				\left(
				4C^{\frac12}\left(\frac{\partial^2F}{\partial{}z^2}+\frac12F-2\right)
				+24\frac{\partial}{\partial{}z}\left(
				F\frac{\partial{}C^{\frac12}}{\partial{}z}
				\right)
				+sC^{\frac32}
				\right).
			\end{aligned}
		\end{equation}
		Using (\ref{EqnScalSecond}) to eliminate the last term, this is $\mathcal{F}_1(F,C)=0$.
	\end{proof}
	
	The equations (\ref{EqnFourBtFlatEqns}) give four equations for the three unknowns $s$, $F$, $C$, so the system appears to be overdetermined.
	However the next lemma shows the equations of (\ref{EqnFourBtFlatEqns}) are not independent.
	\begin{lemma} \label{LemmaReductionOfBtFlat}
		We have the following relation:
		\begin{equation}
			\small
			\begin{aligned}
				\frac{\partial\mathcal{T}}{\partial{}z}
				=
				\frac{-3t}{2\sqrt{C}}\frac{\partial(sC)}{\partial{}z}\mathcal{F}_1
				+12\frac{\partial{}F}{\partial{}z}\mathcal{F}_2
				-6t\frac{\partial\log(C^3F)}{\partial{}z}\frac{\partial}{\partial{}z}\left(CF\frac{\partial{}s}{\partial{}z}\right).
			\end{aligned}
			\label{EqnDerivBtFlat}
		\end{equation}
		Thus $\mathcal{T}(F,C)$ is a constant of the motion for the $8^{th}$ order system $\mathcal{F}_1(F,C)=0$, $\mathcal{F}_2(F,C)=0$, $\triangle{}s=0$.
	\end{lemma}
	\begin{proof}
		(\ref{EqnDerivBtFlat}) follows from a tedious but elementary computation.
	\end{proof}
	
	\begin{lemma} \label{LemmaBtCSC}
		At all points where $C,F\ne0$, the $8^{th}$ order system 
		\begin{equation}
			\begin{aligned}
				\frac{\partial}{\partial{}z}\left(CF\frac{\partial{}s}{\partial{}z}\right)=0, \quad\quad
				\mathcal{F}_1(F,C)\;=\;0, \quad\quad
				\mathcal{F}_2(F,C)\;=\;0
			\end{aligned} \label{EqnReducedBtSystem}
		\end{equation}
		is critically determined.
		The operator $\mathcal{T}$ is a constant of the motion, and (\ref{EqnReducedBtSystem}) combined with $\mathcal{T}(F,C)=0$ admits a $7$-parameter family of solutions.
		
		Up to homothety, the $U(2)$-invariant $B^t$-flat metrics is 5-dimensional, and the family of $U(2)$-invariant CSC $B^t$-flat metrics is 4-dimensional.
	\end{lemma}
	\begin{proof}
		To ascertain whether the system (\ref{EqnReducedBtSystem}) is critically determined, we examine the coefficients on the derivatives of $s$, $F$, and $C$.
		We find coefficients of the form $FC$, $CF^{-1}$, $C^\frac12$, $C^{-\frac12}$, $FC^{-\frac32}$ and so on.
		Provided $F$ and $C$ remain bounded away from $0$ and $+\infty$, we never find degeneracy at highest order derivatives.
		We conclude that the system (\ref{EqnReducedBtSystem}), which has three unknowns and three equations, remains critically determined when $F,C$ remain non-zero.
		
		We count the degrees of freedom in the solution space.
		The equations $\frac{\partial}{\partial{}z}\left(CF\frac{\partial{}s}{\partial{}z}\right)=0$, $\mathcal{F}_1=0$, and $\mathcal{F}_2=0$ are fourth order in $F$, second order in $C$, and second order in $s$, which makes eight derivatives in total, requiring eight initial conditions.
		Then we restrict to $\mathcal{T}=0$.
		From Lemma \ref{LemmaReductionOfBtFlat}, $\mathcal{T}$ is constant along solutions so is completely determined by the system's initial conditions.
		$\mathcal{T}(F,C)$ is third order in $F$, second order in $C$, and first order in $s$, so $\mathcal{T}=0$ is a single algebraic relationship among the initial conditions, and reduces the solution space from eight dimensions to seven.
		Up to homothety the solution space is therefore $5$-dimensional.
		Requiring $s=const$ is the same as imposing an initial condition of $s_z=0$, so the solution space is reduced by one dimension.
		Thus the CSC $B^t$-flat solution space is 4-dimensional up to homothety.
	\end{proof}
	
	\begin{theorem} \label{ThmNonTrivialBtFlats}
		The ZSC $B^t$-flat solutions, $t\ne\infty$, are precisely the ZSC Bach-flat solutions.
		
		Assume $g$ is $B^t$-flat and conformally extremal, $t\ne0,\infty$.
		Then it is CSC if and only if it is ZSC or Einstein.
		
		{If $t\ne0,\infty$ there exist CSC $B^t$-flat solutions that are not conformally extremal.}
	\end{theorem}
	\begin{proof}
		The CSC $B^t$-flat equations are (\ref{EqnFourBtFlatEqns}) with initial condition $s_z=0$.
		As discussed above, this is a system with 6 degrees of freedom (4 up to homothety).
		First we examine the $s=0$ case.
		In this case $\mathcal{T}=16\mathcal{B}$, so $\mathcal{B}(F,F)=0$ and so the metric is Bach-flat.
		Thus $F$ lies in the 3-parameter family given by Lemma \ref{LemmaBSols}.
		Fixing $F$, $\mathcal{F}_1=0$ gives a 2-parameter family of solutions for $C$ and we obtain the expected 5-parameter solution space of ZSC Bach-flat metrics (which has 3 free parameters up to homothety).
		
		Next assume the metric is CSC $B^t$-flat, $s\ne0$, and $g$ conformally extremal.
		By Proposition \ref{PropExtremal}, $F=\frac12C_1e^{-2z}+C_2e^{-z}+C_3e^z+\frac12C_4e^{2z}$.
		Plugging in this, along with $\frac{\partial{}s}{\partial{}z}=0$ into $\mathcal{F}_2=0$, we obtain
		\begin{equation}
			\left(\frac{\partial^2}{\partial{}z^2}C^{-\frac12}-\frac14C^{-\frac12}\right)\;=\;0. \label{EqnsCcase}
		\end{equation}
		Therefore $C=\frac{e^{-z}}{(C_5+C_6e^{-z})^2}$.
		Plugging this into $\mathcal{F}_1=0$ provides
		\begin{equation}
			\begin{aligned}
				&0\;=\;
				C_5(C_1C_5-C_2C_6)e^{-z}
				+\left(-\frac{s}{24}+C_2C_5^2-2C_5C_6-C_3C_6^2\right)
				+C_6(C_4C_6-C_3C_5)e^z. \label{EqnConfBtScalRelation}
			\end{aligned}
		\end{equation}
		We have the seven unknowns $C_1,C_2,C_3,C_4,C_5,C_6,$ and $s$, and (\ref{EqnConfBtScalRelation}) contributes three relations so we have a 4-parameter solution space.
		We consider the possibilities.
		First, the expression for $C$ makes it impossible that $C_5$ and $C_6$ are both zero.
		If $C_5\ne0,C_6=0$ then $C=C_5^{-2}e^{-z}$ so the metric is K\"ahler with respect to $J^+$, and (\ref{EqnConfBtScalRelation}) forces $C_1=0$, $C_2=\frac{s}{24C_5^{2}}$.
		Then $0=\mathcal{T}$ is
		\begin{equation}
			0\;=\;\mathcal{T}
			\;=\;-\frac12e^{2z}s\left(3st-4e^{2z}(1-3t)C_3C_5^2\right),
		\end{equation}
		and because $t\ne0$, this forces $s=0$, contradicting the assumption $s\ne0$.
		(Similarly, assuming $C_5=0,C_6\ne0$ also gives $s=0$, again contradicting $s\ne0$.)
		
		Therefore both $C_5,C_6\ne0$.
		Then (\ref{EqnConfBtScalRelation}) forces $C_1C_5-C_2C_6=0$, $C_4C_6-C_3C_5=0$, and by Proposition \ref{PropEinstCondition} the metric is Einstein.
		We conclude that if a CSC $B^t$-flat metric is conformally extremal, it is ZSC or Einstein.
		
		Finally we prove that some CSC $B^t$-flat metrics are not conformally extremal.
		The family of Einstein solutions (not modding by homothety) is 4-dimensional, and therefore, by what we just proved, the family of CSC $B^t$-flat that are conformally extremal is also 4-dimensional.
		But the space of CSC $B^t$-flat metrics is 6-dimensional, so we conclude that CSC $B^t$-flat metrics exist that are not conformally extremal.
	\end{proof}

	\subsection{Summary} \label{SubsecCanonicalTable}
	
	We summarize our findings in the following chart.
	All of the metrics we have discussed have $F$ of the form $F=1+\frac12C_1e^{-2z}+C_2e^{-z}+C_3e^z+\frac12C_4e^{2z}$, the exceptions being the $\delta{}W^\pm=0$ and $B^t$-flat metrics, whose forms cannot be expressed easily.

	{
		\small
		\begin{longtable}{llll}
			Condition      & Conditions & Conf. & Scalar \\
			on the Metric  & on coefs.          & Factor & Curvature \\
			\hline
			\\
			\endhead
			Extremal K\"ahler & None & $C_0e^{-z}$ & $-24\frac{C_2+C_1e^{-z}}{C_0}$ \\
			\\
			CSC-K & $C_1=0$ & $C_0e^{-z}$ & $-24C_0^{-1}C_2$ \\
			& \multirow{4}{*}{\shortstack[l]{$C_1C_5-C_2C_6=0$, \\ $C_3C_5-C_4C_6=0$ }} & &  \multirow{4}{*}{\shortstack[l]{$-24(C_2C_5^2$\\$\;\;-2C_5C_6+C_3C_6^2)$}}\\
			Einstein &  & \hspace{-0.2in}$\frac{e^{-z}}{\left(C_5+C_6e^{-z}\right)^2}$ & \\
			\\
			K\"ahler-Einstein & $C_1=C_3=0$ & $C_0e^{-z}$ & $-24C_0^{-1}C_2$ \\
			&\multirow{4}{*}{\shortstack[l]{\\$C_1C_5-C_2C_6=0$, \\ $C_3C_5-C_4C_6=0$ \\ \hspace{-0.225in}$C_2C_5^2-2C_5C_6+C_3C_6^2=0$ }} \\
			\multirow{2}{*}{Ricci-Flat} &  & \hspace{-0.2in}$\frac{e^{-z}}{\left(C_5+C_6e^{-z}\right)^2}$ & 0 \\
			\\
			\\
			$W^+=0$ & $C_1=C_2=0$ & \textit{any} & \;\;$0$, if K \\
			\\
			$W^-=0$ & $C_3=C_4=0$ & \textit{any} & $-24\frac{C_2+C_1e^{-z}}{C_0}$, if K \\
			\\
			$\delta{}W=0$ & Identical to Einstein & \textit{any} & \\
			\\
			$\delta{}W^\pm=0$ & \hspace{-0.2in} $e^{\pm\frac32z}(\mathcal{L}^{\pm}(F)-1)\sqrt{C}=C_5$ & \hspace{-0.2in} & \\
			\\
			K\"ahler, $\delta{}W^+=0$ & Identical to CSC-K & $C_0e^{-z}$ & $-24C_0^{-1}C_2$ \\
			\\
			K\"ahler, $\delta{}W^-=0$ & $C_3=0$ & $C_0e^{-z}$ &  $-24\frac{C_1e^{-z}+C_2}{C_0}$ \\
			\\
			K\"ahler, $\delta{}W=0$ & Identical to KE & $C_0e^{-z}$ & $-24C_0^{-1}C_2$ \\
			\\
			Bach-flat &  $C_1C_4-C_2C_3=0$\hspace{0.5in} & \textit{any} &
		\end{longtable}
		\textit{Notes:}
		``K'' or ``K\"ahler'' means K\"ahler with respect to $J^+$; for the $J^-$ case replace $z$ with $-z$ and exchange $W^+$ and $W^-$.
	}

	\section{AmbiK\"ahler Pairs} \label{SubsecModTaubs}
	
	AmbiK\"ahler pairs are from \cite{ACG16}, where they were studied in connection with toric manifolds.
	An \textit{ambiK\"ahler structure} on a manifold is a pair of K\"ahler manifolds $(M^n,J_1,g_1)$ and $(M^n,J_2,g_2)$ where the complex stuctures $J_1$ and $J_2$ produce opposite orientations and the K\"ahler metrics $g_1$ and $g_2$ are conformal.
	Either member of the pair the \textit{ambiK\"ahler transform} of the other.
	From Lemma \ref{LemmaClosedJ}, every $U(2)$-invariant metric on a 4-manifold has an ambiK\"ahler structure using $J^\pm$, conformally related by $e^{\pm2z}$.
	
	Consequently the classic K\"ahler metrics---see the chart in Section \ref{SecClassic}---all have ambiK\"ahler transforms.
	Most of these ambiK\"ahler transforms produce nothing interesting.
	The ambiK\"ahler transform of the Burns metric is the Fubini-study metric, for example, and the transforms of the other LeBrun instanton metrics are extremal K\"ahler metrics on weighted projective spaces---we call these the ``modified LeBrun metrics'' on these weighted projective spaces (these metrics were found by Bryant in \S2.2 of \cite{Bry01}, although their ambiK\"ahler relationship with the LeBrun instantons was unkown).
	The transform of an odd Hirzebruch surface is precisely itself.
	The transforms of the Taub-NUT-$\Lambda$ and Eguchi-Hanson-$\Lambda$ metrics have curvature singularities.
	
	However two cases are more interesting: the Taub-NUT and Taub-bolt.
	The Taub-NUT is hyperK\"ahler with complex structures $I^-$ and its left-translates.
	By Proposition \ref{PropIntegrability} we have
	\begin{equation}
		F\;=\;(1-e^{-z})^2, \quad C\;=\;\frac{C_0e^{-z}}{(1-e^{-z})^2}. \label{EqnTaubNUTBeforeAmbiK}
	\end{equation}
	(In the Appendix we compute $F$ and $C$ explicitly from the classic expression; see also Propositions \ref{PropIntegrability} and \ref{EqnTwoRFEMetrics}.)
	The coordinate range is $z\in(0,\infty]$, the nut is located at $z=\infty$, and the ALF end is located at $z=0$ (Figure \ref{FigTaubNUTs}).
	Separate from this, we have an ambiK\"ahler structure given by complex structures $J^-$ and $J^+$ and respective conformal factors $C_0e^{z}$ and $C_0e^{-z}$ in place of the $C$ of (\ref{EqnTaubNUTBeforeAmbiK}).
	These give the conformal orbit of the classic Taub-NUT three different canonical metrics: itself which is hyperK\"ahler, a 2-ended complete scalar-flat K\"ahler metric, and a complete extremal K\"ahler metric.
	We call the latter two the \textit{modified Taub-NUT metrics of the first and second kinds.}
	
	The modified Taub-NUT of the first kind has complex structure $J^-$ and conformal factor $C=C_0e^{z}$, which gives it the same orientation as the original Taub-NUT by Lemma \ref{LemmaRighInvStructs}.
	This metric is two-ended: the nut at $z=-\infty$ becomes an ALE end, and the ALF end at $z=0$ becomes a cusp-like end.
	This complete, 2-ended metric is scalar flat: using Proposition \ref{PropCurvatureIntro} one computes $s=0$.
	Second, letting $J^+$ be the complex structure and choosing conformal factor $C=C_0e^{-z}$ produces the modified Taub-NUT of the second kind.
	This metric is one-ended: it still has a ``nut'' at $z=\infty$, but the conformal change turns the ALF end into a cusp-like end.
	By Theorem \ref{PropExtremal} it is an extremal K\"ahler metric. It has scalar curvature $s=48(1-e^{-z})$, which is positive and approaches $0$ asymptotically along the cusp.
	
	\begin{figure}[ht]
		\includegraphics[scale=0.31]{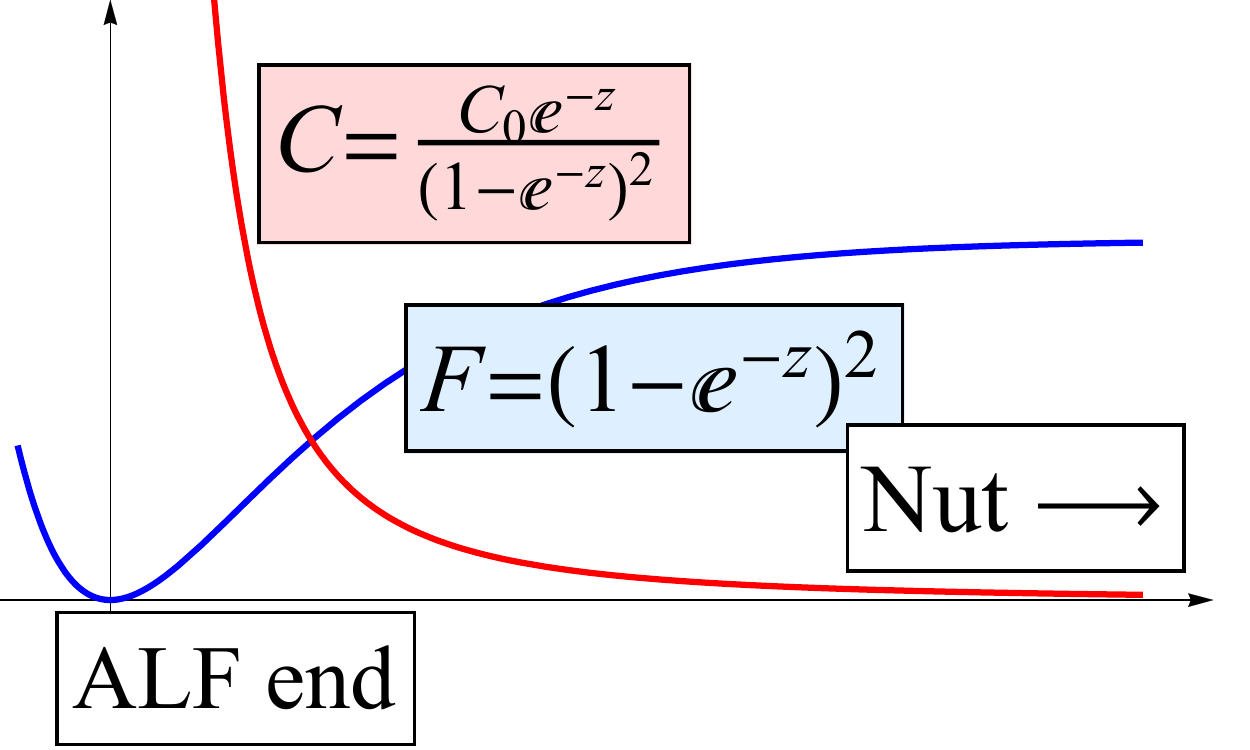} \hspace{0.1in}
		\includegraphics[scale=0.31]{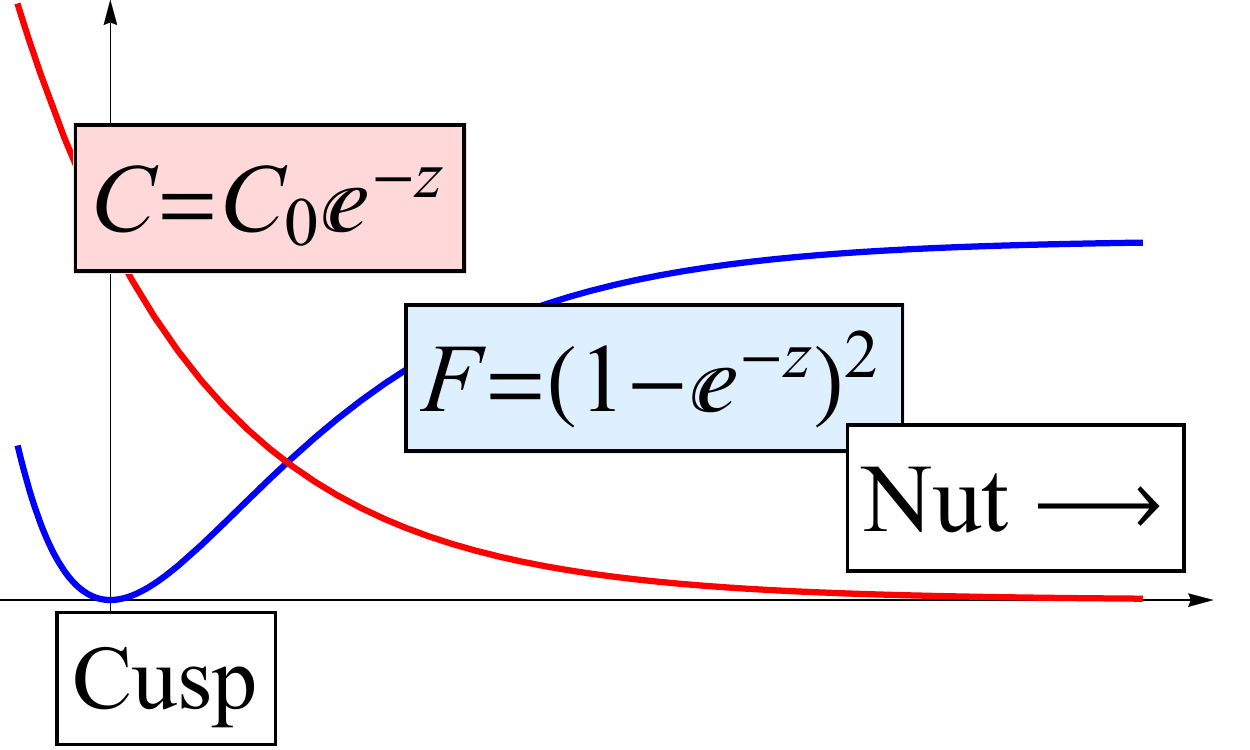} \hspace{0.1in}
		\includegraphics[scale=0.31]{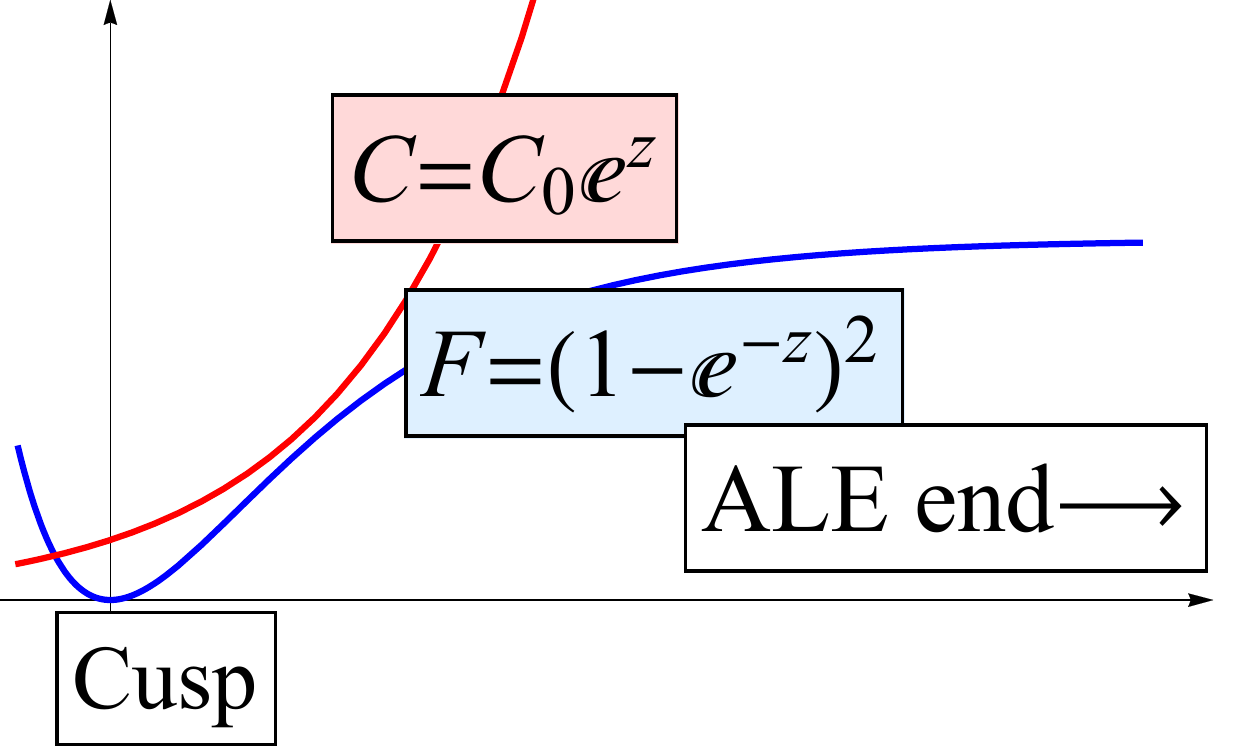}
		
		\caption{
			\it The Taub-NUT and the modified Taub-NUTs of the first and second kinds.
			\label{FigTaubNUTs}
		}
	\end{figure}

	These modified Taub-NUT metrics have both been discovered already, although their conformal relationship with the classic Taub-NUT has not been uncovered until now.
	The modified Taub-NUT of the first kind on $\mathbb{C}^2\setminus\{(0,0)\}$ is the ZSC-K metric of \cite{FYZ} for $n=2$, and the modified Taub-NUT of the second kind is a complete Bochner-flat metric of the type considered in \S2.2 of \cite{Bry01}; see also \cite{TachLiu70}.
	
	\begin{figure}[ht]
		\includegraphics[scale=0.31]{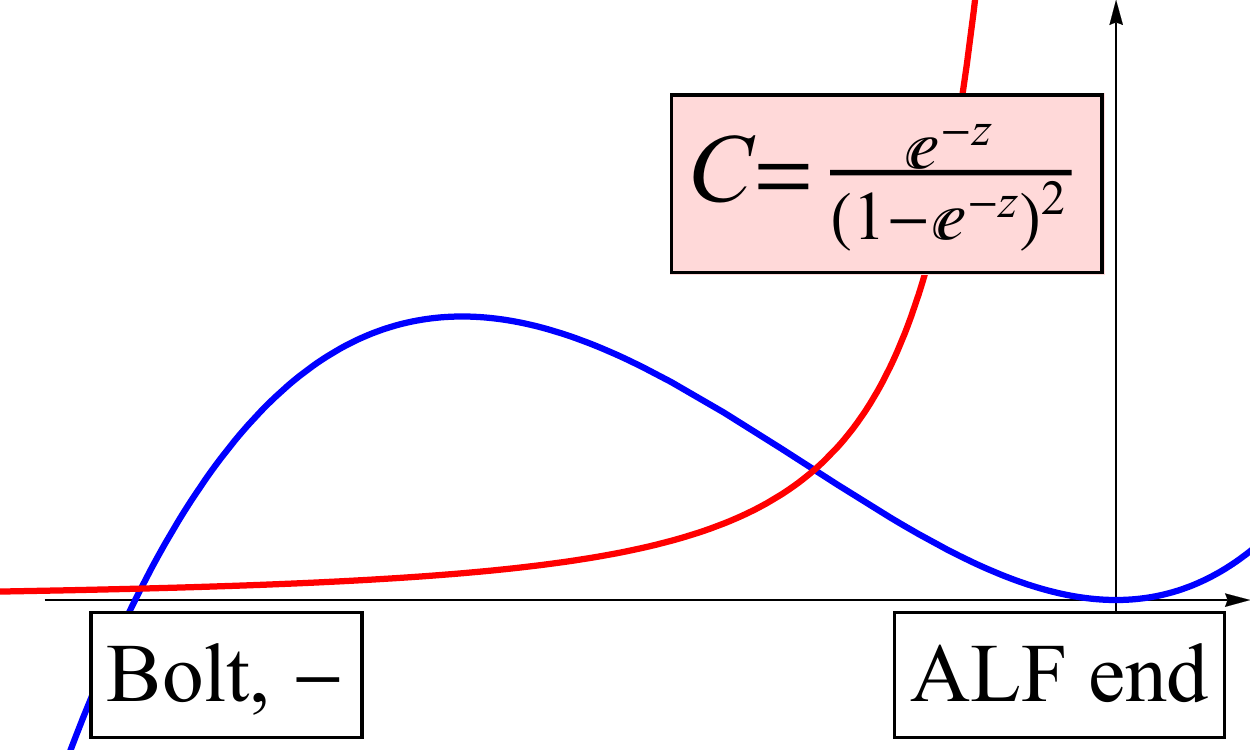}
		\includegraphics[scale=0.31]{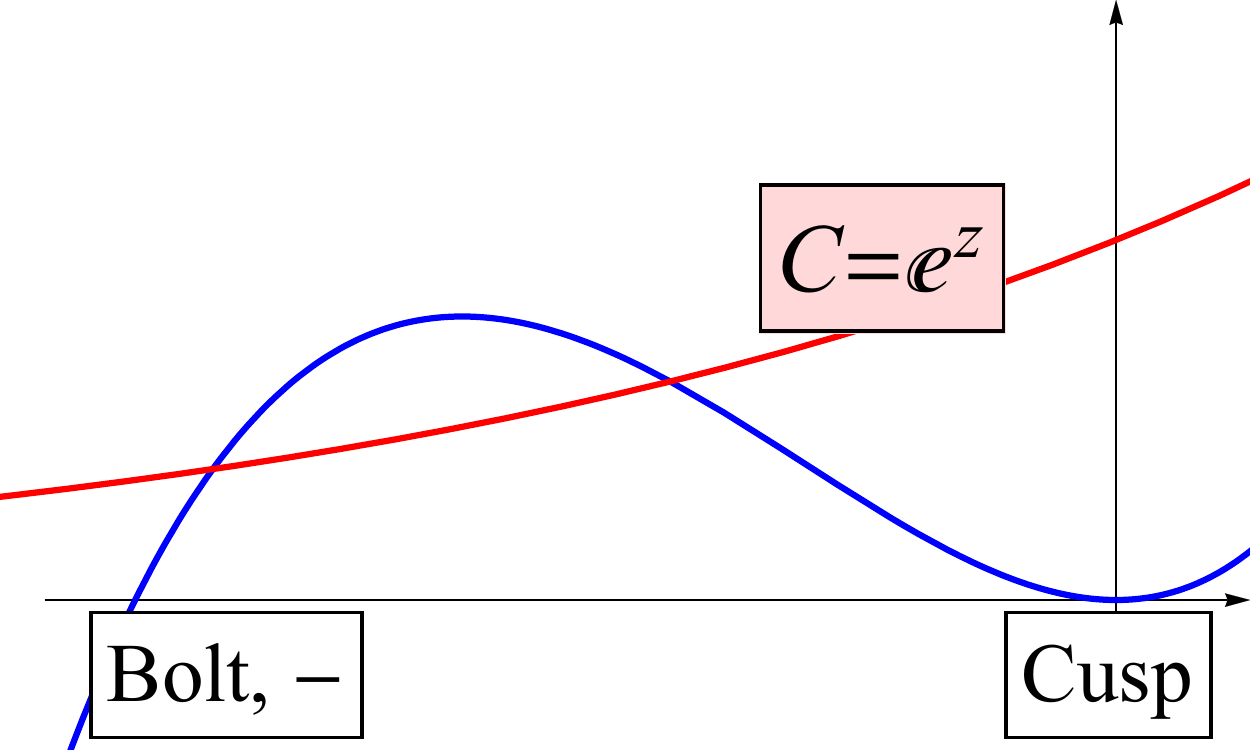}
		\includegraphics[scale=0.31]{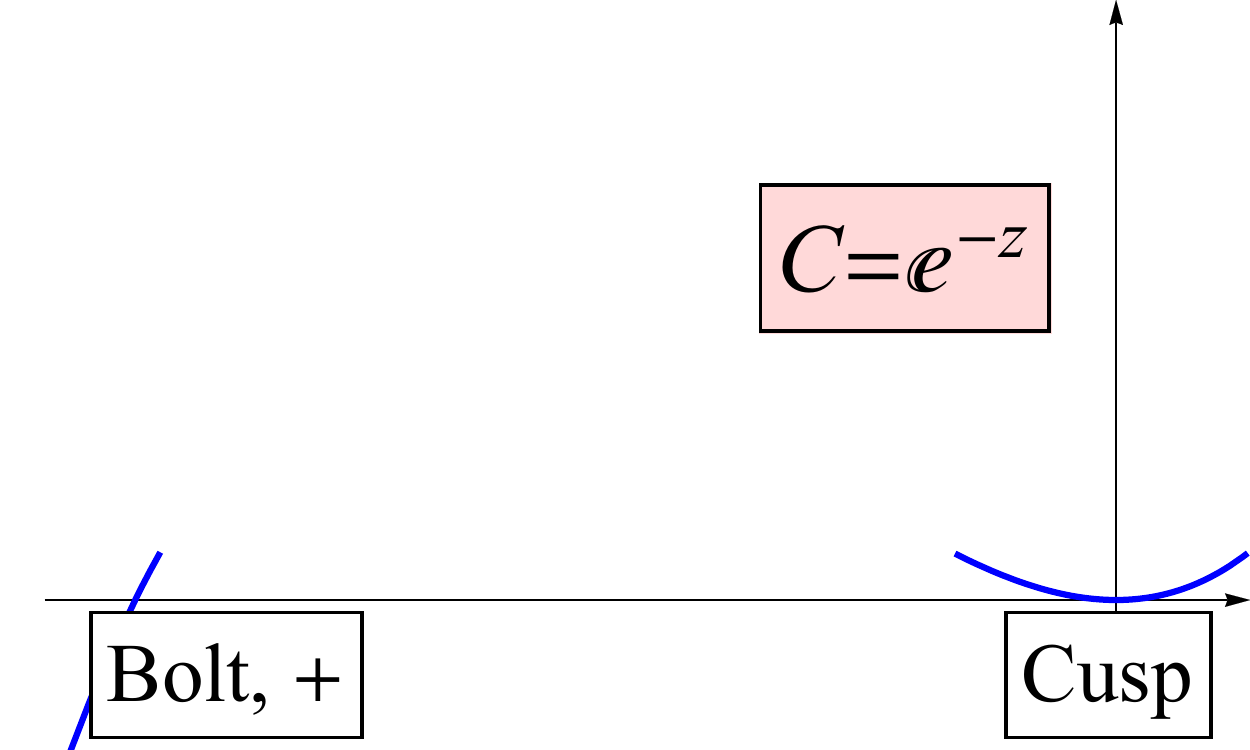}
		
		\caption{
			\it The first image is the Taub-bolt, and the second and third are the modified Taub-bolts of the first and second kinds.
		} \label{FigTaubBolts}
	\end{figure}
	
	The Taub-bolt is Ricci-flat but not K\"ahler (and certainly not hyperK\"ahler) with respect to to any complex stucture\footnote{If it were K\"ahler with respect to any complex structure, whether a complex structure considered here or not, Derdzhinski's theorem would imply it is half-conformally flat which it is not.} (this is possible in the non-compact but not the compact case).
	The metric is
	\begin{equation}
		C\;=\;\frac{C_0e^{-z}}{(1-e^{-z})^2}, \quad\quad
		F\;=\;1-\frac18e^{-2z}+\frac14e^{-z}-\frac94e^z+\frac98e^{2z}
	\end{equation}
	on $z\in[-\log(3),0)$.
	This metric is complete, Ricci-flat, Bach-flat, but not half-conformally flat: both $W^+$ and $W^-$ are non-zero by Proposition \ref{PropHCFmetrics}.
	It has an ALF end at $z=0$ and a bolt of self-intersection $-1$ at $z=-\log(3)$.
	The underlying manifold is the total space of $\mathcal{O}(-1)$.
	It is conformally K\"ahler with respect to either $J^-$ or $J^+$, creating an ambiK\"ahler pair---the \textit{modified Taub-bolt metrics of the first and second kinds,} respectively.
	Changing between $J^-$ and $J^+$ reverses the orientation, so changes the self-intersection number of the bolt from $-1$ to $+1$.
	
	With the complex structure $J^-$ and conformal factor $C=C_0e^z$ we obtain an extremal K\"ahler metric we call the \textit{modified Taub-bolt of the first kind}.
	This metric continues to have a bolt of self-intersection $-1$ at $z=-\log(3)$, but the ALF end at $z=0$ has been transformed into a cusp-like end.
	The scalar curvature is $s=54C_0^{-1}(1-e^{z})$, which is positive and approaches $0$ along the cusp.
	Its underlying complex manifold is the total space of $\mathcal{O}(-1)$.
	Its ambiK\"ahler transform has complex structure $J^+$ and conformal factor $C=C_0e^{-z}$; we call this extremal K\"ahler metric the \textit{modified Taub-bolt of the second kind}.
	The orientation has been reversed and the bolt has self-intersection $+1$ at $z=-\log(3)$.
	The ALF end at $z=0$ has been conformally transformed into a cusp-like end.
	The scalar curvature is $s=6C_0^{-1}(-1+e^{-z})$, which again is positive and approaches zero asymptotically along the cusp.
	Its underlying complex manifold is the total space of $\mathcal{O}(+1)$.
	The modified Taub-bolt of the second kind is the only complete extremal K\"ahler metric on any surface with a rational curve of positive self-intersection that is known to the authors.

	\appendix
	
	\section{The Classic Metrics} \label{SecClassic}
	
	Numerous $U(2)$-invariant 4-metrics have been developed, the basic method going back to (at least) 1916 \cite{Schw16}.
	In Euclidean signature, most of these metrics were developed in the late 1970's with notable works by Plebanski-Demianski \cite{PD76}, Eguchi-Hanson \cite{EH78}, Gibbons-Hawking \cite{GH78}, Page \cite{Page78a} \cite{Page78b}, Gibbons-Pope \cite{GPo78}, and Gibbons-Perry \cite{GPe80}.
	The LeBrun metrics \cite{LeB88} appeared slightly later.
	See also \cite{EGH80} and \cite{BeBer82}.
	
	\subsection{Transcribing the classic metrics} \label{SubSecClassic}
	
	Fitting the classic metrics into the present framework is straightforward.
	As an illustration we consider the Taub-NUT metric, classically
	\begin{equation}
		g\;=\;\frac14\frac{r+m}{r-m}dr^2
		+4m^2\frac{r-m}{r+m}(\eta^1)^2
		+(r^2-m^2)\left((\eta^2)^2+(\eta^3)^2\right)
	\end{equation}
	where $r\in[m,\infty)$.
	Solving $-dz=\frac{2\sqrt{AB}}{C}dr$ for $z$, we obtain coordinate changes
	\begin{equation}
		\begin{aligned}
			z\;=\;-\log\frac{r-m}{r+m}\quad\quad\text{and}\quad\quad
			r\;=\;\frac{m(1-e^{-z})}{1+e^{-z}}
		\end{aligned}
	\end{equation}
	where $z\in(0,\infty]$.
	Then $C$ and $F=\frac{B}{C}$ are
	\begin{equation}
		\begin{aligned}
			&C=(r^2-m^2)=\frac{4m^2e^{-z}}{(1-e^{-z})^2}
			\;\;\text{and}\;\;
			F=\frac{4m^2}{(r+m)^2}
			=(1-e^{-z})^2.
		\end{aligned}
	\end{equation}
	Similarly one may transcribe any of the classic $U(2)$-invariant metrics in this way.
	We list below the classic metrics and their expressions via conformal factor $C$ and a function $F$ of the form $F=1+\frac12C_1e^{-2z}+C_2e^{-z}+C_3e^{z}+\frac12C_4e^{2z}$.
	
	{
		\small
		\begin{longtable}{lllllll}
			& \hspace{-0.15in}Classic Coefs  & Conf.   & $F$ or coefs. & Special & Underlying \\
			Name & $A,B,C$        & Factor & of $F$ & Metric & Manifold   \\
			\hline
			\endhead
			\\
			\nopagebreak\multirow{3}{*}{\footnotesize Taub-NUT} & $\frac14\frac{r+m}{r-m}$ &       & & RF, HK &  \\
			\nopagebreak& $4m^2\frac{r-m}{r+m}$ & \hspace{-0.05in}$\frac{4m^2e^{-z}}{(1-e^{-z})^2}$ & $(1-e^{-z})^2$ & HCF & $\mathbb{C}^2$ \\
			\nopagebreak& $r^2-m^2$ & & & C-Extr & \\
			\\
			\nopagebreak Modified &   &       & & Extr-K &  \\
			\nopagebreak \; Taub-NUT, & See Sec. \ref{SubsecModTaubs}. & $e^{z}$ & $(1-e^{-z})^2$ & HCF & $\mathbb{C}^2$ \\
			\nopagebreak \; First Kind& & & & C-RF \\
			\\
			\nopagebreak & & & & \\
			\nopagebreak Modified & & & & SFK & \\
			\nopagebreak\; Taub-NUT, & {See Sec. \ref{SubsecModTaubs}.} & {$e^{-z}$} & $(1-e^{-z})^2$ & HCF & $\mathbb{C}^2\setminus\{0\}$ \\
			\nopagebreak\; Second Kind & & & & C-RF \\
			\\
			\nopagebreak\multirow{4}{*}{Taub-NUT-$\Lambda$} & $\frac{r^2-m^2}{4\Delta}$ & \hspace{-0.2in}\multirow{4}{*}{$\frac{4m^2e^{-z}}{(1-e^{-z})^2}$} & \hspace{-0.2in}$\frac{m-L+\frac{1}{3}m^3\Lambda}{m}$, & & $O(-k)$ \\
			\nopagebreak & $\frac{4m^2\Delta}{r^2-m^2}$ & & \hspace{-0.25in}$-\frac{m-L+\frac{1}{3}m^3\Lambda}{m}$, & Einst-$\Lambda$ & sometimes,\\
			\nopagebreak & $r^2-m^2$ & & \hspace{-0.25in}$-\frac{m+L-\frac{1}{3}m^3\Lambda}{m}$, & BF & but usually \\
			\nopagebreak & $\Delta=*$ & & \hspace{-0.2in}$\frac{m+L-\frac{1}{3}m^3\Lambda}{m}$ & C-Extr & singular\\
			\\
			\nopagebreak \multirow{4}{*}{Taub-bolt} & \hspace{-0.05in}$\frac{r^2-m^2}{r^2-\frac52mr+m^2}$ & &  \\
			\nopagebreak & \hspace{-0.25in}$\frac{16m^2(r^2-\frac52mr+m^2)}{r^2-m^2}$ & \hspace{-0.15in}$\frac{16m^2e^{-z}}{(1-e^{-z})^2}$ & \hspace{-0.15in}$-\frac14,\frac14,-\frac94,\frac94$ & BF & $O(-1)$ \\
			\nopagebreak & \hspace{-0.05in}$4(r^2-m^2)$ & & & RF \\
			\\
			\nopagebreak Modified & & & & Extr-K & $O(-1)$ \\
			\nopagebreak \;Taub-bolt, & See Sec. \ref{SubsecModTaubs}. & $e^z$ & \hspace{-0.15in}$-\frac14,\frac14,-\frac94,\frac94$ & BF \\
			\nopagebreak \;first kind & & & & C-RF \\
			\\
			\nopagebreak Modified & & & & Extr-K & $O(+1)$ \\
			\nopagebreak \;Taub-bolt, & See Sec. \ref{SubsecModTaubs}. & $e^{-z}$ & \hspace{-0.15in}$-\frac14,\frac14,-\frac94,\frac94$ & BF \\
			\nopagebreak \;second kind & & & & C-RF \\
			\\
			\nopagebreak\multirow{3}{*}{Burns} & $(1-\frac{m^2}{r^2})^{-1}$ & & \\
			\nopagebreak & $r^2(1-\frac{m^2}{r^2})$ & $e^{z}$ & \hspace{-0.15in}$0,-m^2,0,0$ & SFK & $O(-1)$ \\
			\nopagebreak & $r^2$ & & & \\
			\\
			\nopagebreak & $(1-\frac{m^4}{r^4})^{-1}$ & & \\
			\nopagebreak Eguchi & $r^2(1-\frac{m^4}{r^4})$ & $e^{z}$ & \hspace{-0.15in}$-2m^4,0,0,0$ & RF & $O(-2)$ \\
			\nopagebreak \quad -Hanson & $r^2$ \\
			\\
			& $A^{-1}$ & & \hspace{-0.15in}$-2m^4(k-1),$ & SFK \\
			\nopagebreak{LeBrun} & $r^2A$ & $e^{z}$ & \hspace{-0.05in}$m^2(k-2)$, & RF iff & $O(-k)$ \\
			\nopagebreak & $r^2$ & & \hspace{-0.05in}$0$, $0$ & \;\;$k=2$ \\
			\\
			\nopagebreak & & & \hspace{-0.15in}$-2m^4(k-1)$, & Extr-K, & One point\\
			\nopagebreak Modified & See Sec. \ref{SubsecModTaubs} & $e^{-z}$ & \hspace{-0.05in}$m^2(k-2),$ & CSC iff & compact. \\
			\nopagebreak \hspace{0.1in}LeBrun & & & \hspace{-0.05in}$0$, $0$ & \;\;$k=1$ & of $O(+k)$\\
			\\
			\nopagebreak & \hspace{-0.2in} $(1-\frac{m^4}{r^4}-\frac{\Lambda}{6}r^2)^{-1}$ & & & & \hspace{-0.1in}$O(-k)$ iff \\
			\nopagebreak Eguchi& \hspace{-0.0in}$r^2-\frac{m^4}{r^2}-\frac{\Lambda}{6}r^4$ & $e^{z}$ & $-2m^4$, $0$, & \hspace{-0.1in}Einst-$\Lambda$ & \hspace{-0.1in}$m^4=\frac{4(1+k)}{3}$, \\
			\nopagebreak \;\;-Hanson-$\Lambda$ & $r^2$ & & $-\frac16\Lambda$, $0$ & & \hspace{-0.1in}$\Lambda=4-2k$, \\
			\nopagebreak & & & & & \hspace{-0.1in}and $k\ge2$ \\
			\\
			\nopagebreak\multirow{3}{*}{Fubini-Study} & $\left(1+\frac{\Lambda}{6}r^2\right)^{-2}$ & & \\
			\nopagebreak & $r^2\left(1+\frac{\Lambda}{6}r^2\right)^{-2}$ & $6e^{-z}$ & \hspace{-0.0in}$0,-\Lambda,0,0$ & Einst-$\Lambda$ & $\mathbb{C}P^2$ \\
			\nopagebreak & $r^2\left(1+\frac{\Lambda}{6}r^2\right)^{-1}$ \\
			\\
			\nopagebreak\multirow{4}{*}{Page} & & & $\approx-.2442$ & \\
			\nopagebreak & \multirow{2}{*}{See \cite{ACG03}} & \hspace{-0.2in}$\frac{14.931e^{-z}}{\Lambda(1+e^{-z})^2}$, & $\approx-.2442$ & Einst-$\Lambda$ & $\mathbb{C}P^2\#\overline{\mathbb{C}P^2}$ \\
			\nopagebreak & & \hspace{-0.2in}\textit{(approx.)} & $\approx-.2442$ & BF & \\
			\nopagebreak & & & $\approx-.2442$ & C-Extr
		\end{longtable}
		\noindent\textit{Notes:} 
		In the Taub-NUT-$\Lambda$ metric,
		$\Delta=(r^2-L)^2+\frac{\Lambda}{4}\left(-\frac13r^4+2m^2r^2+m^4\right)$.
		In the LeBrun metrics, $A=\left(1-\frac{m^2}{r^2}\right)\left(1+(k-1)\frac{m^2}{r^2}\right)$ and $k=1$ gives the Burns metric and $k=2$ gives the Eguchi-Hanson metric.
		In the Page metric, the value of the coefficients is $C_1=C_2=C_3=C_4=\frac{\sinh(z_0)-\cosh(z_0)}{(2+\cosh(2z_0))\sinh(z_0)}$ and $z_0$ is the unique real solution of $e^{4z_0}-4e^{z_0}-3=0$.
	}

	\subsection{The Page Metric} \label{SubsecPageMetr}
	
	The Page metric, which won't fit on the chart, has classic coefficients
	\begin{equation}
		\begin{aligned}
			&A(r)=\frac{3(1+\nu^2)}{\Lambda}\left(\frac{1-\nu^2r^2}{(3-\nu^2-\nu^2(1+\nu^2)r^2)}\frac{1}{1-r^2}\right) \\
			&B(r)=\frac{3(1+\nu^2)}{\Lambda}\left((1-r^2)\frac{(3-\nu^2-\nu^2(1+\nu^2)r^2)}{(3+\nu^2)^2(1-\nu^2r^2)}\right), \\
			&C(r)=\frac{3(1+\nu^2)}{\Lambda}\left(\frac{1-\nu^2r^2}{\nu(3+\nu^2)}\right)
		\end{aligned}
	\end{equation}
	where $\nu$ is the positive root of $\nu^4+4\nu^3-6\nu^2+12\nu-3=0$, about $\nu\approx.28$.
	Any other choice both disrupts the bolting condition and also creates a non-canonical metric.
	In this paper's formulation, $C$ and $F$ are
	\begin{equation}
		\begin{aligned}
			C&\;=\;\frac{12(1+\nu^2)}{\Lambda\nu(3+\nu^2)}\frac{e^{-z}}{(1+e^{-z})^2} \quad\quad \text{and} \\
			F&\;=\;
			\frac{-\nu^4+6\nu^2+3}{4\nu(3+\nu^2)}
			-\frac{(1-\nu^2)^2}{4\nu(3+\nu^2)}\left(\cosh(2z)
			+2\cosh(z)\right). \label{EqnPageMetric}
		\end{aligned}
	\end{equation}
	The stated choice of $\nu$ makes the constant $\frac{-\nu^4+6\nu^2+3}{4\nu(3+\nu^2)}$ equal to $1$.
	
	The Page metric is not K\"ahler, but is conformal to one of the extremal K\"ahler metrics constructed by Calabi \cite{Cal82} on Hirzebruch surfaces.
	We mentioned the ``odd'' Hirzebruch sufaces $\Sigma_{2k-1}$ in Section \ref{SubSubsecBoltsNuts}; using our techniques we can construct many extremal metrics on each $\Sigma_{2k-1}$.
	To do so let $C=C_0e^{-z}$ so the metric is K\"ahler with respect to $J^+$.
	For a compact manifold, $F$ must reach $0$ at two places: after translating in $z$ we may assume $F(-z_0)=F(z_0)=0$, some $z_0>0$.
	We require a bolt of self-intersection $+k$ at $-z_0$ and $-k$ at $+z_0$, so by Lemma \ref{LemmaBolting} we have the four conditions $F(-z_0)=0$, $F(z_0)=0$, $F'(-z_0)=k$, and $F'(z_0)=-k$.
	Using $F=1+\frac12C_1e^{-2z}+C_2e^{-z}+C_3e^{z}+\frac12C_4e^{2z}$, for every choice of $z_0>0$, $k\in\mathbb{N}$, $C_0>0$ we obtain a unique smooth metric
	\begin{equation}
		\begin{aligned}
			&C=C_0e^{-z}, \quad
			F\;=\;1+C_1'\cosh(2z)\,+\,2C_2'\cosh(z), \quad\;\text{where} \\
			&C_1'=\frac{\sinh(z_0)-k\cosh(z_0)}{(2+\cosh(2z_0))\sinh(z_0)}, 
			\quad
			C_2'=\frac{-2\sinh(2z_0)+k\cosh(2z_0)}{2(2+\cosh(2z_0))\sinh(z_0)}.
		\end{aligned} \label{EqnCoefsHirzes}
	\end{equation}
	By Lemma \ref{LemmaClosedJ} these are each extremal K\"ahler metrics.
	For fixed $k$ (which fixes the Hirzebruch surface $\Sigma_{2k-1}$), we have two choices---$C_0$ and $z_0$---giving a 2-parameter family of extremal metrics on each surface.
	Since the K\"ahler cone is two-dimensional---parameterized by the mass of the two bolts---it is easy to see that we have found an extremal K\"ahler metric in each K\"ahler class.
	(We remark that it is already known that every K\"ahler class on $\Sigma_{2k-1}$ has an extremal representative; see \cite{HS95} or \cite{HS97}.)
	
	We remark that an ambiK\"ahler transform of any Hirzebruch surface is itself.
	This is because $F$ of (\ref{EqnCoefsHirzes}) is invariant under $z\mapsto-z$, whereas the conformal factor $e^{-z}$ and complex structure $J^+$ switch: $e^z$ becomes $e^{-z}$ and $J^+$ becomes $J^-$.
	Therefore the map $z\mapsto-z$ is not only a biholomorphism but an isometry between ambiK\"ahler pairs.
	
	\begin{corollary} \label{ThmBachFlatHirz}
		Assume the $U(2)$-invariant Riemannian manifold $(M^4,g)$ is compact and Bach-flat.
		Then $M^4=\mathbb{C}P^2\sharp\overline{\mathbb{C}P}{}^2$ and $g$ is conformal to the Page metric.
	\end{corollary}
	\begin{proof}
		By Proposition \ref{PropCurvatureIntro} the Bach-flat metrics are those with $F=1+\frac12C_1e^{-2z}+C_2e^{-z}+C_3e^z+\frac12C_4e^{2z}$ and $C_1C_4=C_2C_3$.
		Ignoring the condition $C_1C_4=C_2C_3$ for the moment, up to homothety the metric given by (\ref{EqnCoefsHirzes}) encodes all possibilities for compact, extremal $U(2)$-invariant manifolds.
		In particular when $k=1$ the surface $\Sigma_{2k-1}$ is $\Sigma_{1}=\mathbb{C}P^2\sharp\overline{\mathbb{C}P}{}^2$.
		
		Next we impose the Bach-flat condition $C_1C_4=C_2C_3$; for (\ref{EqnCoefsHirzes}) this becomes $C_1'=C_2'$, which is the same as
		\begin{equation}
			k\;=\;\frac{2(1+2\cosh(z_0))\sinh(z_0)}{2\cosh(z_0)+\cosh(2z_0)}. \label{EqnCalabiBachFlatCond}
		\end{equation}
		The expression on the right is increasing for $z_0\ge{}0$, is zero at $z_0=0$ and approaches $2$ as $z_0\rightarrow\infty$.
		Therefore the only $k\in\mathbb{N}$ for which we can solve (\ref{EqnCalabiBachFlatCond}) is $k=1$.
		For this $k=1$, a unique positive value of $z_0$ solves (\ref{EqnCalabiBachFlatCond})---this is the positive root of $e^{4z_0}-4e^{z_0}-3=0$ which works out to be about $z_0\approx0.579$.
		
		We conclude that one unique function $F$ gives a Bach flat metric on a compact $U(2)$-invariant manifold: the function $F$ of (\ref{EqnCoefsHirzes}) where $k=1$ and $z_0$ is the unique positive solution of $e^{4z_0}-4e^{z_0}-3=0$.
		Because the Page metric is $U(2)$-invariant and Bach-flat, certainly the $F$ of (\ref{EqnPageMetric}) equals the $F$ of (\ref{EqnCoefsHirzes}) for these choices (as can be verified by direct computation).
	\end{proof}

\end{document}